\documentclass[reqno,10pt,a4paper]{amsart}
\usepackage[latin1]{inputenc}
\usepackage[english]{babel}
\usepackage{amsmath,amssymb}

\newcommand{\bbZ} {{\mathbb{Z}}}   
\newcommand{\bbN} {{\mathbb{N}}}   
\newcommand{\bbR} {{\mathbb{R}}}   
   
\newcommand{\bbC} {{\mathbb{C}}}

\newcommand{\cal}{\mathcal}

\newcommand{\all}[2]{ \left \{\, {#1} \, : \, {#2} \, \right \} }
\newcommand{\alll}[2]{ \{\, {#1} \, : \, {#2} \,  \} }
\newcommand{\set}[1]{{\{#1\}}}
\newcommand{\abs}[1]{\left | {#1}\right | }
\newcommand{\norm}[1]{\left\| {#1}\right\|}
\newcommand{\floor}[1]{\lfloor {#1}\rfloor}
\newcommand{\clo}[1]{{\overline{#1}}}
\newcommand{\Jac}{\mathrm{Jac}}

\theoremstyle{plain}
\newtheorem{Theorem}{Theorem}[section]
\newtheorem{Proposition}[Theorem]{Proposition}
\newtheorem{Lemma}[Theorem]{Lemma}
\newtheorem{Corollary}[Theorem]{Corollary}

\theoremstyle{definition}

\newtheorem{Remark}[Theorem]{Remark}

\newtheorem*{KProblement*}{Kolmogorov's Problem for the $\epsilon$-entropy}
\newtheorem*{KProblemwidth*}{Kolmogorov's Problem for $m$-widths}
\newtheorem*{ZConjecture*}{Zakharyuta's Conjecture}

\title{New solution of a problem of Kolmogorov on width asymptotics in holomorphic function spaces}

\author{Oscar F.~Bandtlow$^{\dag}$}
\address{Oscar F.~Bandtlow, School of Mathematical Sciences,
Queen Mary University of London, London E3 4NS,
United Kingdom}
\email{o.bandtlow@qmul.ac.uk}

\author{St\'ephanie Nivoche$^{\dag\dag}$}
\address{St\'ephanie Nivoche, CNRS and Laboratoire J.-A. Dieudonn\'e U.M.R. 7351, Universit\'e C\^ote d'Azur, Parc Valrose
06108 Nice Cedex 02, France}
\email{Stephanie.NIVOCHE@univ-cotedazur.fr}

\thanks{$^{\dag}$The research of OFB was supported by the EPSRC grant EP/R012008/1. $^{\dag\dag}$The research of SN was supported by ANR grant ``ANR-14-CE34-0002-01'' for the project ``Dynamics and CR geometry''.} 
            
\date{29 April 2020}

      
\begin{document}

\begin{abstract}
Given a domain $D$ in $\mathbb{C}^n$ and $K$ a compact subset of $D$, the set 
$\mathcal{A}_K^D$ of all restrictions of functions holomorphic on $D$ the modulus of which is 
bounded by $1$ is a compact subset of the Banach space $C(K)$ of continuous functions on $K$. 
The sequence $(d_m(\mathcal{A}_K^D))_{m\in \mathbb{N}}$ of Kolmogorov $m$-widths of 
$\mathcal{A}_K^D$ provides a measure of the degree of compactness of the set 
$\mathcal{A}_K^D$ in $C(K)$ and the study of its asymptotics has a long history, essentially going 
back to Kolmogorov's work on $\epsilon$-entropy of compact sets in the 1950s. 
In the 1980s Zakharyuta showed that for suitable $D$ and $K$ the 
asymptotics 
\begin{equation}
\label{eq:KP}
 \lim_{m\to \infty}\frac{- \log d_m(\mathcal{A}_K^D)}{m^{1/n}} = 
 2\pi \left ( \frac{n!}{C(K,D)}\right ) ^{1/n}\,,
\end{equation}
where $C(K,D)$ is the Bedford-Taylor relative capacity of $K$ in $D$
is implied by a conjecture, now known as Zakharyuta's Conjecture,
concerning the approximability of the regularised relative extremal
function of $K$ and $D$ 
by certain pluricomplex Green functions. Zakharyuta's Conjecture was proved by 
Nivoche in 2004 thus settling (\ref{eq:KP}) at the same time. \\
We shall give a new proof of the asymptotics 
(\ref{eq:KP}) for $D$ strictly hyperconvex and $K$ non-pluripolar which does not rely on Zakharyuta's Conjecture. Instead we proceed 
more directly by a two-pronged approach establishing sharp upper and lower bounds for the 
Kolmogorov widths. The lower bounds follow from concentration results of independent interest  
for the eigenvalues of a 
certain family of Toeplitz operators, while the upper bounds follow from an application of the 
Bergman-Weil formula together with an exhaustion procedure by special holomorphic polyhedra. 
\end{abstract}

\maketitle

\newpage

\tableofcontents

\newpage

\section{Introduction}

\subsection{Kolmogorov's problem}
\label{sec:1.1}

In the 1930s approximation
theory received a new impetus, when Kolmogorov \cite{Ko1}
introduced the concept 
of the width (or diameter) of a compact set: 
given a normed linear space $(X, \norm{\cdot})$ and  $C\subset X$ a
compact subset, then, for any $m\in \bbN$, the \emph{Kolmogorov
  $m$-width} of $C$ in $X$ is the quantity
\[ d_m(C,X)=\mathop{\vphantom{p}\inf}\limits_{\dim L < m}\sup_{x\in C}
\mathop{\vphantom{p}\inf}\limits_{y\in L}\|x-y\|\,,\] 
where the outermost infimum is taken over subspaces $L$ of $X$. 
The sequence $(d_m(C,X))_{m\in \mathbb{N}}$ provides a measure of how well the compact subset $C$ 
can be approximated by finite-dimensional subspaces of $X$. 

In the 1950s, Kolmogorov returned to approximation theory through his work on complexity theory, 
such as his study \cite{Ko} of Vitushkin's work on Hilbert's 13th problem about the complexity
of function spaces, in which he proved that the space of 
analytic functions of $n$ variables is ``larger" than the space 
of analytic functions of $m$ variables when $n>m$. 
Drawing inspiration from 
Shannon's information theory \cite{shannon48}, Kolmogorov \cite{Ko2}
introduced the concept of the $\epsilon$-entropy of a compact set $C$ in a metric space $X$.   
Given any $\epsilon>0$, there is a 
covering of $C$ by subsets of $X$ with diameters not exceeding $2\epsilon$. 
Denoting the smallest cardinality of such a covering by 
$N_\epsilon(C,X)$, the \emph{$\epsilon$-entropy} $H_\epsilon(C,X)$ of 
$C$ in $X$ is given by\footnote{Here `$\log$' denotes the natural logarithm, 
whereas in the original definition the logarithm with basis $2$ was used. 
} 
\[ H_\epsilon(C,X)=\log N_\epsilon(C,X)\,.\]
The family $(H_\epsilon(C,X))_{\epsilon>0}$ thus quantifies the degree of 
compactness of $C$. 

The determination of entropies and widths 
of classes of functions has several purposes. 
Firstly, it can produce new invariants making it possible to 
distinguish and classify function sets in infinite dimensional
spaces. The meaning of the fundamental concept
of ``number of variables'' 
often manifests itself in this way. Secondly, computations of widths
and entropies foster the creation of new methods of approximation. 
Thirdly, it produces stimuli for computational mathematics 
by giving directions for the creation of the most expedient algorithms
to solve practical problems. 
A more thorough discussion of this circle of ideas can be found in the original papers 
\cite{Ti2, K-T, Mi}, in the monograph \cite{Ti5}, or 
in the selected works of Kolmogorov 
on Information Theory \cite{Ko93}, containing a reprint of \cite{K-T}. 

For a fixed compact set $C$ in $X$, the computation of the exact value of 
$H_\epsilon(C,X)$ for each $\epsilon>0$ is a rather difficult task. 
The determination of the 
corresponding asymptotic order as $\epsilon$ tends to zero, however, turns out to be more manageable  
even in infinite-dimensional spaces. 
The first results in this direction are due to 
Kolmogorov and Tikhomirov \cite{K-T}, who 
determined the asymptotic order 
of the $\epsilon$-entropy of analytic functions of $n$ variables defined on 
a bounded domain in $\bbR^n$ extending analytically 
to some domain in $\bbC^n$.

More precisely, suppose we are given a domain $D$ in $\bbC^n$ 
containing a compact set $K$. Write $H^\infty(D)$ for the Banach space of all 
bounded, holomorphic functions in $D$ endowed with the sup-norm and 
$C(K)$ for the space of all continuous functions 
on the compact set $K$ equipped with the sup-norm. By Montel's theorem, 
the set ${\cal A}_K^D$ of restrictions of functions in the unit ball of 
$H^\infty(D)$ is a compact subset of $C(K)$. Kolmogorov and Tikhomirov \cite{K-T}
showed that, under natural assumptions on $D$ and $K$, the 
asymptotic order of 
$H_\epsilon({\cal A}_K^D)= H_\epsilon({\cal A}_K^D, C(K))$  
is  $\log(\epsilon^{-1})^{n+1}$, that is,  
there is a constant $M>1$ such that 
\[ M^{-1} < \frac{H_\epsilon({\cal A}_K^D)}{\left ( \log(\epsilon^{-1}) \right)^{n+1}} < M 
\quad (\forall \epsilon >0)\,.
\]
The existence of the limit
\[ \lim_{\epsilon \to 0} \frac{H_\epsilon({\cal A}_K^D)}{\left ( \log(\epsilon^{-1}) \right)^{n+1}}\,,\]
and its precise value, however, remained open. 

In the one-dimensional case, the problem of showing the existence of the limit was raised by 
Kolmogorov (see \cite[p 134]{Ko93}) and solved in the late 1950s, with various generalisations  
throughout the 1960s and 1970s. 
The quantity central to this quest turned out to be  
the capacity $C(K,D)$ of the compact set $K$ relative to the domain $D$, which we now recall. 
Suppose that we can solve the Dirichlet problem on 
$D\setminus K$. Let $u$ be the
relative extremal function for $K$ in $D$, that is, 
the unique harmonic function in $D\setminus K$, continuous on the closure of $D\setminus K$, 
equal to $0$ in 
$\partial D$ and equal to $-1$ on $\partial K$. 
With $\Gamma$ denoting a smooth contour separating $K$ from $\partial D$, and
$n$ the normal to $\Gamma$ directed from $K$ to $\partial
D$, the relative capacity $C(K,D)$ is given by 
\[ C(K,D)=\int_\Gamma \frac{\partial u(z)} {\partial n}\,|dz |\,.\] 
The solution of Kolmogorov's problem in dimension one can now be formulated as follows (see \cite{E0, Ba, L-T, W}): 

{\it Let $\partial D$ have positive logarithmic capacity (that is, $\partial D$ is non-polar) and let $\bbC \setminus D$ have a countable set of connected components. Then 
\[ \lim_{\epsilon \to 0}   \frac{H_\epsilon({\cal A}_K^D)}{\left ( \log(\epsilon^{-1}) \right )^2} 
 =\frac{C(K,D)}{2\pi}\,.\]
}

The solution of Kolmogorov's problem for functions of several variables remained elusive, although 
the solution for one variable makes it possible to solve it for several variables 
in special cases, for instance when $K$ and $D$ are Cartesian products of one-dimensional sets 
(see \cite{Z1,Z2}).

In the 1980s, with the development of pluripotential theory, in particular with the introduction of 
different types of extremal plurisubharmonic functions 
with respect to the complex Monge-Amp\`ere operator, the setting of the 
problem could be formulated precisely. For this we need to recall some definitions and properties relating to 
relative extremal functions and relative capacity.

If $D$ is an open set in $\bbC^n$ and $E$ a subset of $D$, the
{\it relative extremal function} for $E$ in $D$ (see \cite{Z3, Be2, K01, K1, Si, Sa, B-T2}) 
is defined as
\[ u_{E,D}(z)=\sup 
\all{v(z)}{v \text{ is psh on $D$}, \, v|_E \le -1, \, v \le 0}\quad (z \in D)\,.\]
 
 Here and in the following, we write `psh' for `plurisubharmonic',  
a notion that replaces the notion of subharmonicity in one variable. 
It turns out that the upper semicontinuous regularisation $u_{E,D}^*$ of $u_{E,D}$ 
is psh on $D$. 
In one variable, $u_{E,D}^*$ is closely related to harmonic measure. 

In several variables, 
the natural context for the study of this function is provided by hyperconvex domains: 
a domain $D$ in $\bbC^{n}$ is said to be \emph{hyperconvex} 
if there exists a continuous plurisubharmonic exhaustion
function $\varrho :D\to (-\infty,0)$. In one dimension, a domain $D$ is hyperconvex if and only 
if we are able to solve the Dirichlet problem on $D$. 
In several variables, every hyperconvex domain is pseudoconvex (or holomorphically convex) 
and every pseudoconvex domain is the union of an increasing sequence
of hyperconvex domains. Note that if $D$ is an open set and $E$ a non-pluripolar relatively 
compact subset of $D$, then 
$D$ is hyperconvex if and only if for any point $w\in \partial D$ we have 
$\lim_{z \to w} u_{E,D}(z) =0$.

For later use, we note that if $D$ is a bounded hyperconvex open set and $K\subset D$ is a 
compact set, then we say that $K$ is {\it regular} in $D$ if $u^*_{K,D}$ is a
continuous function.

In one complex variable, $u_{K,D}$ is harmonic in $D\setminus K$ and $\Delta u_{K,D}$ is a
positive measure supported on $K$. In several variables, the Laplace operator is replaced 
by the \emph{complex Monge-Amp\`ere operator} $(dd^c)^n$, 
defined as the $n$th exterior power of $dd^c=2i\partial \bar{\partial}$, that is, 
\[ (dd^c)^n= \underbrace{dd^c \land \ldots \land dd^c}_{\text{$n$ times}}\,.\] 
Here, $d=\partial +\bar{\partial}$ and $d^c= i(\bar{\partial}-\partial)$. 

In one variable the complex Monge-Amp\`ere operator
is equal to the Laplace operator $\Delta$, since $dd^c =\Delta\, dx\land dy$ in $\bbR^2$ or $\bbC$.
In several variables, if $D$ is a hyperconvex domain in
$\bbC^n$ containing a compact subset $K$, then $u^*_{K,D}$ 
is maximal on $D \setminus K$ and $(dd^c u^*_{K,D})^n=0$ on $D \setminus K$. In this 
case, the complex Monge-Amp\`ere operator $(dd^c u^*_{K,D})^n$ is well defined and turns out to 
be a positive measure supported on $K$ \cite[Section 4.5]{K2}.

The \emph{relative capacity} of $K$ in $D$ (see \cite{Be1, B-T2}) is defined as 
\[ C(K,D)=\sup\all{\int_K (dd^c u)^n}{u \in PSH(D,(-1,0))}\,,\] 
where $PSH(\Omega, I)$ denotes the set of all psh functions on a domain $\Omega$ in 
$\mathbb{C}^n$ with values in an interval $I\subset [-\infty,+\infty)$.

It turns out that for $D$ a hyperconvex domain in $\bbC^n$ and $K$ a compact subset of $D$, the 
relative extremal function and the relative capacity are related as follows \cite[Section 4.6]{K2}:
\[ C(K,D)=\int_D (dd^c u^*_{K,D})^n = \int_K (dd^c u^*_{K,D})^n\,.\]
Note that $C(K,D)<\infty$ by the Chern-Levine-Nirenberg estimate (see, for example, 
\cite[Proposition~3.4.2]{K2}). Note also that the previous definition of 
$C(K,D)$ in dimension one coincides with the more 
general one given above, since 
\[ \int_\Gamma \frac{\partial u(z)}{\partial n}\,|dz|  = \int_K \Delta u\,dx dy\,\]
by Green's formula.

In the 1980s, using the generalisation of the notion of relative capacity to higher dimensions given 
above, Zakharyuta \cite{Z1} formulated a more precise version of Kolmogorov's problem. 

\begin{KProblement*} For $D$ a domain in $\mathbb{C}^n$ and $K$ a compact subset of $D$ 
show that  
\begin{equation}
\label{eq:KProblement}
\lim_{\epsilon \to 0} 
 \frac{H_\epsilon({\cal A}_K^D)}{(\log(\epsilon^{-1}))^{n+1}}
= \frac{2C(K,D)}{(2\pi)^n (n+1)!}\,.
\end{equation}
\end{KProblement*}

In \cite{Z1} Zakharyuta provided a sketch that Kolmogorov's problem could be solved 
provided that a certain conjecture, now known as Zakharyuta's Conjecture, could be established. 
More detailed accounts of this reduction were provided in \cite{Z2, Z9, Z11}. 
Zakharyuta's Conjecture concerns the approximability of the regularised relative extremal function 
of $K$ and $D$ by pluricomplex Green functions in $D$ with logarithmic poles in $K$. The precise 
setting of this conjecture as well as the definitions and properties of these functions can be 
found in \cite{N3,N4}.

This conjecture was proved in the one-dimensional case by Skiba and Zakharyuta 
\cite{Z-S} and in the multi-dimensional case by Nivoche \cite{N3,N4} under the hypothesis that 
the domain $D$ be bounded and hyperconvex and that the compact set $K$ be regular in $D$, 
a rather natural setting, since in this case,  
the relative extremal function and the pluricomplex Green function are continuous psh functions 
tending to zero on the boundary of the domain.
 
As mentioned earlier, there is a close connection between Zakharyuta's Conjecture and 
Kolmogorov's problem. Indeed, in order to solve Kolmogorov's problem, it is sufficient to prove 
that Zakharyuta's Conjecture is true, as was shown in 
\cite{Z-S} for $n=1$ and in \cite{Z1} (see also \cite{Z9, Z11}) for $n>1$, provided that 
$K$ be regular in $D$ with non-zero Lebesgue measure, and that the 
domain $D$ be strictly hyperconvex, a rather 
natural notion that is defined as follows.

A domain $D \in \mathbb{C}^n$ is said to be 
\emph{strictly hyperconvex} if there exists a bounded domain 
$\Omega$ and a continuous exhaustion function $\varrho \in PSH( \Omega, (-\infty,1))$ 
such that $D=\all{z\in\Omega}{\varrho (z)<0 }$.  Note that any strictly pseudoconvex domain is 
strictly hyperconvex. 

All in all, it follows from \cite{Z1} and \cite{N4} that Kolmogorov's Problem for the 
$\epsilon$-entropy is solved, for $D$ strictly hyperconvex and $K$ regular in $D$ with positive 
Lebesgue measure.

\subsection{A new solution of Kolmogorov's problem}
\label{sec:1.2}

In this paper we will provide a new self-contained solution of Kolmogorov's problem, which does 
not rely on Zakharyuta's Conjecture, but instead proceeds more directly, and, at the same time, 
makes the required assumptions explicit and transparent. 

It turns out that Kolmogorov's Problem 
on the asymptotics of $H_\epsilon( \mathcal{A}_K^D)$ as $\epsilon$ tends to zero is 
in fact equivalent to the following problem on the asymptotics of the Kolmogorov widths 
$d_m(\mathcal{A}_K^D)=d_m(\mathcal{A}_K^D,C(K))$ as $m$ tends to infinity. 

\begin{KProblemwidth*} For $D$ a domain in $\mathbb{C}^n$ and $K$ a
  compact subset of $D$ 
show that 
\begin{equation}
\label{eq:KProblemwidth}
\lim_{m\to \infty} \frac{-\log d_m({\cal A}_K^D)}{m^{1/n}} = 
  2\pi\left(\frac{ n!}{C(K,D)}{}\right)^{1/n}\,.
\end{equation}
\end{KProblemwidth*}
The two problems of Kolmogorov are equivalent in the sense that if the 
limit (\ref{eq:KProblement}) exists then so 
does the limit (\ref{eq:KProblemwidth}) and \textit{vice versa}. For a complete proof of this fact, 
see \cite{Z10}, which uses ideas from \cite{Mi} and \cite{L-T}. 

Our approach to solve Kolmogorov's Problem for $m$-widths, 
and hence for the $\epsilon$-entropy, will proceed as follows. 
In Section~\ref{sec:2}, we will establish sharp lower bounds for the Kolmogorov widths by 
studying the eigenvalue distribution of a family of Toeplitz operators defined on a family of 
Bergman spaces. Unfortunately this method does not appear to provide sharp upper bounds. 
As a result, another method will be used in Section~\ref{sec:3} to establish sharp upper bounds 
for the Kolmogorov widths, first in the special case where $D$ and $K$ are special holomorphic 
polyhedra using the Bergman-Weil formula, then in the general case exploiting refinements of 
approximation arguments from \cite{N4}. Curiously enough, this method does not seem to 
provide sharp lower bounds. 

All in all, we shall establish the following result. 
\begin{Theorem} \label{theor10} 
Let $D$ be a domain in $\mathbb{C}^n$ and $K$ a compact subset of $D$. 
\begin{itemize}
\item[(i)] If $D$ is strictly hyperconvex and $K$ is non-pluripolar, 
then  
\[ \limsup_{m \to \infty}  \frac{-\log d_m({\cal A}_K^D)}{m^{1/n}} \le 2\pi \left (\frac{ n!}{C(K,D)}
\right )^{1/n}\,.\]
\item[(ii)] If $D$ is bounded hyperconvex, then 
\[ \liminf_{m \to \infty}  \frac{-\log d_m({\cal A}_K^D)}{m^{1/n}} \ge 2\pi \left (\frac{ n!}{C(K,D)}
\right )^{1/n}\,.\]
\end{itemize}
In particular, if $D$ is strictly hyperconvex and $K$ is non-pluripolar, then 
\[ \lim_{m \to \infty}  \frac{-\log d_m({\cal A}_K^D)}{m^{1/n}} = 2\pi \left (\frac{ n!}{C(K,D)}
\right )^{1/n}\,.\]
\end{Theorem}

\subsection{Preliminaries and notation} 
\label{sec:1.3}
For $D$ an open subset of $\mathbb{C}^n$ and $K$ a compact subset
of $D$ we write 
\begin{align*}
\mathcal{O}(D)&=\all{f:D\to \mathbb{C}}{\text{$f$ holomorphic on
    $D$}}\,,\\
\mathcal{O}(K)&=\all{f:K\to \mathbb{C}}{\text{$f$ holomorphic on
    a neighbourhood of $K$}}\,.
\end{align*}

We use $H^\infty(D)$ for the Banach space of bounded holomorphic
functions on $D$ equipped with the supremum norm on $D$ and 
$A(K)$ for the completion of $\mathcal{O}(K)$ in the Banach space
$C(K)$ of continuous functions on $K$. 

For the derivation of our bounds on the Kolmogorov widths
$d_m({\cal A}_K^D)$ in Sections~\ref{sec:2} and \ref{sec:3} it will be
convenient to introduce the following generalisation of the Kolmogorov
widths to bounded operators on a Banach space. 

For $T:X\to Y$ a bounded operator with $X$ and $Y$ Banach spaces, 
we associate with it the sequence
$(d_m(T))_{m\in \bbN}$ of \emph{Kolmogorov numbers} given by  
\[ d_m(T)=\mathop{\vphantom{p}\inf}\limits_{\dim L < m}\sup_{\|x\|_X\leq 1}
\mathop{\vphantom{p}\inf}\limits_{y\in L}\|Tx-y\|_Y\,,
\] 
where the outermost infimum is taken over subspaces $L$ of
$Y$. It turns out that $\lim_{m\to \infty}d_m(T)=0$ if and only if $T$
is compact (see, for example, \cite[Proposition~2.2.1]{CS}), so the
sequence $(d_m(T))_{m\in \mathbb{N}}$ provides a measure of the degree of
compactness of $T$. 

In order to link the two notions observe that if we define 
the \emph{canonical mapping} 
\[ J:H^\infty(D) \to A(K) \]
\[ Jf=f|_K \]
then $\mathcal{A}_K^D$ is the image under $J$ of the unit ball $B_{H^\infty(D)}$
in $H^\infty(D)$ and 
\[ d_m(J)= d_m(J(B_{H^\infty(D) }),A(K)) 
=d_m(\mathcal{A}^D_K) 
\quad (\forall m \in \mathbb{N})\,,\] 
so in order to investigate Kolmogorov's Problem on $m$-widths it
suffices to study the asymptotics of the Kolmogorov numbers of the canonical
mapping $J:H^\infty(D) \to A(K)$. Note that the canonical mapping is
in fact an embedding if $D$ is a domain and $K$ is a set of uniqueness (which is, for example, the case if $K$ is non-pluripolar).  

For later use, we associate two more sequences with the bounded operator 
$T:X\to Y$ between Banach spaces. 
The sequence $(c_m(T))_{m\in \mathbb{N}}$ of \emph{Gelfand 
numbers} 
\[ c_m(T)=\mathop{\vphantom{p}\inf}\limits_{\operatorname{codim} L< m}
\, \sup_{\substack{x\in L\\ \|x\|_X\leq 1}}\|Tx\|_Y\,,\]
where the infimum is taken over closed subspaces $L$ of $X$, which, like the sequence 
of Kolmogorov numbers, also quantifies compactness of~$T$, and the 
sequence  $(a_m(T))_{m\in \bbN}$ of \emph{approximation numbers} 
given by 
\[ a_m(T)=\inf \all{\|T-F\|_{X\to Y}}{\text{$F:X\to Y$ with $\mathrm{rank}(F)<m$}}\,,\]
which quantifies the degree of approximability of $T$ by operators of
finite rank. It turns out that (see, for example, \cite[Equation~(2.2.12) and Proposition~2.4.6]{CS}) 
\begin{equation}
\label{eq:dmam}
d_m(T) \le a_m(T) \le \sqrt{2m} d_m(T)\quad (\forall m \in
\mathbb{N})\,,
\end{equation}
and that (see, for example, \cite[Equation~(2.3.15) and Proposition~2.3.4]{CS}) 
\begin{equation}
\label{eq:cmam}
c_m(T) \le a_m(T) \le \sqrt{2m} c_m(T)\quad (\forall m \in
\mathbb{N})\,.
\end{equation}
which implies that the Kolmogorov numbers decay at a certain
stretched exponential speed precisely if the approximation numbers 
and, in turn, the Gelfand numbers do, the speeds being the same in all cases. 

We also note that all three sequences enjoy the
following submultiplicativity property. If $W$ and $Z$ are Banach spaces and
$S:W\to X$ and $R:Y\to Z$ are bounded operators, then  
\begin{equation}
\label{eq:dmsubm}
s_m(RTS)\leq \|R\|_{Y\to Z} s_m(T) \| S\|_{W\to X}\quad (\forall
m\in \mathbb{N})\,,
\end{equation} 
where $\|R\|_{Y\to Z}$ and $\|S\|_{W\to X}$ denote the operator norms
of $R$ and $S$, respectively, and where $s_m(T)$ denotes any of the Kolmogorov, 
Gelfand or approximation numbers (see, for example,
\cite[Theorems~2.3.3, 2.4.3, 2.5.3 ]{pietsch87}). 

Finally, we note that in the particular case where $X$ and $Y$ are
Hilbert spaces, then the 
Kolmogorov numbers, Gelfand numbers and approximation numbers  
coincide (see, for example, \cite[Theorem
2.11.9]{pietsch87}). 

\section{Lower bound for the Kolmogorov widths}
\label{sec:2}

In the original formulation of Kolmogorov's problem of the asymptotics of the $m$-widths
 $d_m({\cal A}_K^D)$, all spaces are equipped with the supremum norm. 
In the case where $D$ is a strictly hyperconvex domain containing a holomorphically 
convex 
and regular compact subset $K$ of positive Lebesgue measure,  Zakharyuta 
\cite{Z1,Z2} (see also 
\cite{Z9,Z11}, for a more recent exposition) and  Aytuna \cite{Ay1} have developed a 
rather sophisticated theory 
which guarantees that the asymptotics of Kolmogorov widths 
of natural spaces of holomorphic functions associated with $D$ and $K$ coincide, including Banach spaces equipped with the supremum norm as well as Hilbert spaces 
with a weighted $L^2$-norm, where the weight arises from a bounded psh function on $D$. Indeed, in this case, there exist pairs 
$(H_0,H_1)$ of Hilbert spaces which are adherent to the pair 
$(A(K),H^\infty(D))$ with norms weaker than the supremum norm 
(see for instance Lemma $4.12$ and Corollary $4.13$ in \cite{Z11}).

In Subsections~\ref{sec:2.1}, \ref{sec:2.2} and \ref{sec:2.3}, 
we will develop an $L^2$-approach for the asymptotics of Kolmogorov 
$m$-widths for bounded hyperconvex $D$ and $K\subset D$ compact with positive 
Lebesgue measure,  which relies 
on studying the eigenvalue asymptotics of a sequence of compact Toeplitz operators on 
certain 
Bergman spaces, which in turn relies on asymptotic bounds for Bergman kernels. 
In Subection~\ref{sec:2.4}, we shall 
connect the $L^2$-bounds with the usual supremum norm bounds in the original formulation of 
Kolmogorov's problem, under the slightly stronger hypothesis that $D$ be strictly 
hyperconvex. We will not rely on the theory developed by Zakharyuta and Aytuna, preferring instead to use simple bounds 
for Kolmogorov widths coupled with approximation arguments to obtain sharp lower 
bounds for the $m$-widths $d_m({\cal A}_K^D)$.
Finally, in Subsection~\ref{sec:2.5} we will generalise the result from Subsection~\ref{sec:2.4} to allow $K$ to be non-pluripolar.   

\subsection{Spectral asymptotics of Toeplitz operators} 
\label{sec:2.1}
Let $D$ be a domain in $\mathbb{C}^n$ and let $L^\infty(D)$ denote the Banach space of
complex-valued essentially bounded functions on $D$ equipped with the essential supremum 
norm $\norm{\cdot}_{\infty}$. 
For $\varphi:D\to \mathbb{R}$ a bounded measurable function, 
we write $L^2_\varphi(D)$ for the weighted $L^2$-space of 
Lebesgue measurable functions on $D$ equipped with the norm 
\[ \norm{f}_{L^2_\varphi(D)}^2=\int_D\abs{f}^2e^{-2\varphi}\,dm=
\int_D\abs{f}^2\,dm_\varphi\,, \]
where $m$ denotes $2n$-dimensional Lebesgue measure on $\mathbb{C}^n$, and $dm_\varphi=e^{-2\varphi}\,dm$.

The corresponding weighted Bergman space will be denoted by $H^2_\varphi(D)$, that is, 
\[ H^2_\varphi(D)=\all{f\in L^2_\varphi(D)}{\text{$f$ holomorphic on $D$}}\,. \]
Note that since $\varphi$ is bounded, the weighted Bergman space $H^2_\varphi(D)$ and the 
(unweighted) standard Bergman space $H^2_0(D)$, or simply $H^2(D)$, are isomorphic as Banach spaces. 
Thus, point-evaluation $f\mapsto f(z)$ is continuous on $H^2_\varphi(D)$ for every $z\in D$ and 
$H^2_\varphi(D)$ is a reproducing kernel Hilbert space, the kernel of which we denote by 
$B_{D,\varphi}$ or simply $B_{\varphi}$ if the domain $D$ is understood. Thus, 
$B_\varphi:D\times D \to \mathbb{C}$ with 
\[ f(z) = \int_D B_\varphi(z,\zeta)f(\zeta)\,dm_\varphi(\zeta) \quad 
(\forall f\in H^2_\varphi(D), \forall z\in D)\,.\]
In particular, for each $\zeta \in D$, the function 
$z\mapsto B_\varphi(z,\zeta)$ is holomorphic and 
$\overline{B_\varphi (z,\zeta)}=B_\varphi (\zeta,z)$. Moreover, the reproducing kernel can 
be written 
\[ B_\varphi(z,\zeta)=\sum_me_m(z)\overline{e_m(\zeta)}\,,\]
where $(e_m)$ is an orthonormal basis of $H^2_\varphi(D)$ and the sum converges 
uniformly on compact subsets of $D\times D$. 

Using the Bergman kernel $B_\varphi$, the orthogonal projection
\[ P_\varphi:L^2_\varphi(D)  \twoheadrightarrow H^2_\varphi(D)\,,\]
known as Bergman projection in this context, can be written 
\[ P_\varphi f(z) = \int_{D} B_\varphi (z,\zeta) f(\zeta)\,dm_\varphi(\zeta)\,.\]
Note that if $J_\varphi$ denotes the natural embedding of $H^2_\varphi(D)$ in 
$L^2_\varphi(D)$ then $P_\varphi$ is the adjoint of $J_\varphi$, that is, 
\[ P_\varphi^*=J_\varphi\,.\]

For $\chi$ in $L^\infty(D)$ we write 
$M_\chi$ for the corresponding multiplication operator on $L^2_\varphi(D)$, that is, 
\[ M_\chi:L^2_\varphi(D) \to L^2_\varphi(D) \]
\[ M_\chi f=\chi \cdot f\,.\]

Note that $M_\chi$ is bounded with operator norm $\|M_\chi\|=\|\chi\|_{\infty}$ and 
that $M_\chi$ is the zero operator precisely when the support of 
$\chi$ is a Lebesgue null set. Ultimately, this is the reason why we require $K$ to 
have non-zero Lebesgue measure for this and the following three subsections. 

The compression of  $M_\chi$ to $H^2_\varphi(D)$, denoted by $T_{\chi,\varphi}$, 
is known as Toeplitz operator with symbol $\chi$ in this context, that is, 
\[ T_{\chi,\varphi}:H^2_\varphi(D)\to H^2_\varphi(D) \]
\[ T_{\chi,\varphi}=P_\varphi M_\chi J_\varphi \,. \]
Clearly, we have for $f\in H^2_\varphi(D)$
\[ T_{\chi,\varphi}f(z)= 
\int_{D} B_\varphi (z,\zeta) \chi(\zeta) f(\zeta)\,dm_\varphi(\zeta)\,.\]

We shall now collect some properties of Toeplitz operators on $H^2_\varphi(D)$ which are 
fairly standard but difficult to find in the literature in the stated generality. We start with 
positivity. 

\begin{Lemma}
\label{lem:positivity}
If $\chi \in L^\infty(D)$ is real-valued and non-negative then $T_{\chi,\varphi}$ is a bounded, 
self-adjoint, positive operator. In particular, 
the spectrum of $T_{\chi,\varphi}$ is contained in $[0,\|\chi\|_{\infty}]$.  
\end{Lemma}

\begin{proof}
Since $M_\chi$ is bounded, the operator $T_{\chi,\varphi}$ is also bounded.  
Moreover, as $\chi$ is real-valued and non-negative, there is a real-valued 
$\psi \in L^\infty(D)$ with $\chi=\psi^2$. Thus $M_\psi= M_\psi^*$, and we have 
\[ T_{\chi,\varphi}=P_\varphi M_\psi^2 J_\varphi=(M_\psi J_\varphi)^*M_\psi J_\varphi\,. \]
Hence $T_{\chi,\varphi}$ is self-adjoint and positive, and its spectrum is contained in $[0,\infty)$. 
Moreover, as $\norm{T_{\chi,\varphi}}\leq \norm{M_\chi}=\|\chi\|_{\infty}$ the remaining assertion follows. 
\end{proof}

Next we turn to compactness properties of $T_{\chi,\varphi}$.

\begin{Lemma}
\label{lem:trace}
If $\chi\in L^\infty(D)$ has compact support in the open set $D\subset \mathbb{C}^n$, 
then there are positive constants $c_1$ and $c_2$, such that 
\begin{equation}
\label{eq:stretchedexpo}
d_m(T_{\chi,\varphi}) \leq c_1 \exp( -c_2 m^{1/n}) \quad (\forall m \in \mathbb{N})\,.
\end{equation}
In particular, $T_{\chi,\varphi}$ is trace class
with 
\begin{equation}
\label{eq:trace:1}
{\rm Tr}(T_{\chi,\varphi})= \int_DB_{\varphi}(z,z)\chi(z)\,dm_\varphi(z)\,,
\end{equation}
and 
\begin{equation}
\label{eq:trace:2}
 {\rm Tr}(T_{\chi,\varphi}^2)= \int_D\int_D\abs{B_{\varphi}(z,\zeta)}^2\chi(z)\chi(\zeta)
\, dm_\varphi(z)\,dm_\varphi(\zeta)\,.
\end{equation}
\end{Lemma}

\begin{proof}
Since $\chi$ has compact support we can choose a domain $U$ containing the support of 
$\chi$ such that the closure of $U$ is a compact subset of $D$. We can now write 
\begin{equation}
\label{eq:Tfac}
T_{\chi,\varphi}=P_\varphi M_\chi J_\varphi=P_\varphi \widetilde{M}_\chi J_{U}J_{D,U}\,, 
\end{equation}
where $J_{D,U}:H^2_\varphi(D)\to H^2_\varphi(U)$ and 
$J_{U}:H^2_\varphi(U)\to L^2_\varphi(U)$ denote the canonical embeddings and 
$\widetilde{M}_\chi$ denotes the operator of multiplication by $\chi$, 
albeit considered as an operator from 
$L^2_\varphi(U)$ to $L^2_\varphi(D)$. 
Note that $\widetilde{M}_\chi$ is well defined and bounded since the support of $\chi$ 
is contained in $U$.

Now, since $U$ is compactly contained in $D$ standard arguments (see, for example, 
\cite[Theorem~4.7]{BJ08}), show that the Kolmogorov numbers of $J_{D,U}$ enjoy a stretched 
exponential 
bound, that is, there are positive constants 
$\tilde{c}_1$ and $\tilde{c}_2$ such that 
\begin{equation}
\label{eq:stretchedexpo_tilde}
d_m(J_{D,U}) \leq \tilde{c}_1 \exp( -\tilde{c}_2 m^{1/n}) \quad (\forall m \in \mathbb{N})\,.
\end{equation}
Moreover, using the factorisation (\ref{eq:Tfac})
and the fact that  $P_\varphi \widetilde{M}_\chi J_{U}$ is bounded we have 
\[ d_m(T_{\chi,\varphi}) \leq 
\|P_\varphi \widetilde{M}_\chi J_{U}\| d_m(J_{D,U}) \quad (\forall m \in \mathbb{N})\,, \]
which, together with (\ref{eq:stretchedexpo_tilde}), yields (\ref{eq:stretchedexpo}). In particular, the 
operator $T_{\chi,\varphi}$ is trace class, since its Kolmogorov numbers, and hence its singular 
values are summable.

Given an orthonormal basis $(e_m)$ of $H^2_\varphi(D)$ we then have, using properties 
of the reproducing kernel stated earlier, 
\begin{align*}
{\rm Tr}(T_{\chi,\varphi}) & = \sum_m \int_D\int_D B_{\varphi}(z,\zeta)\chi(\zeta)e_m(\zeta)
\overline{e_m(z)}
\,dm_\varphi(\zeta)\,dm_\varphi(z) \\
& = \int_D\int_D B_{\varphi}(z,\zeta)\chi(\zeta)B_\varphi(\zeta,z) \,dm_\varphi(\zeta)\,dm_\varphi(z) \\
& = \int_D B_{\varphi}(\zeta,\zeta) \chi(\zeta)\,dm_\varphi(\zeta)\,,
\end{align*}
and (\ref{eq:trace:1}) is proven. For the proof of the second trace formula we observe that 
we can also write $T^2_{\chi,\varphi}$ as an integral operator 
\begin{equation*}
T^2_{\chi,\varphi}f(z)=\int_D K_{\chi,\varphi}(z,\zeta)f(\zeta)\,dm_\varphi(\zeta)\,,
\end{equation*}
with kernel 
\[ K_{\chi,\varphi}(z,\zeta)=
\int_D B_{\varphi}(z,\zeta') \chi(\zeta') B_{\varphi}(\zeta',\zeta) \chi(\zeta) 
\,dm_\varphi(\zeta')\,,  
\]
so using the same arguments as before, we have
\[{\rm Tr}(T^2_{\chi,\varphi}) =
\int_D\int_D B_{\varphi}(\zeta',\zeta)\chi(\zeta)B_{\varphi}(\zeta,\zeta')\chi(\zeta')
\,dm_\varphi(\zeta)\,dm_\varphi(\zeta') \]
and (\ref{eq:trace:2}) follows by using the symmetry of the reproducing kernel.
\end{proof}

Suppose now that the symbol $\chi$ is non negative with compact support in $D$. By the 
preceding two lemmas, the associated Toeplitz operator $T_{\chi,\varphi}$ is positive and 
compact. If $f$ is an eigenfunction of $T_{\chi,\varphi}$ with positive eigenvalue $\lambda$, then 
\[ 
\lambda \| f \|^2_{H^2_{\varphi}(D)} = 
(T_{\chi,\varphi}f,f)_{H^2_{\varphi(D)}}=
(M_{\chi} J_\varphi f, J_\varphi f)_{L^2_\varphi(D)}
=\int_D \chi \abs{f}^2\,dm_\varphi\,.
\]
Thus, if $\chi$ is the characteristic function of a compact subset $K$ of $D$, the eigenvalue 
$\lambda$ measures the concentration of the mass of the corresponding eigenfunction 
$f$ to $K$. \\
In order to investigate this further we shall write 
$(\lambda_m(T_{\chi,\varphi}))_{m\in \mathbb{N}}$ for the 
corresponding eigenvalue sequence arranged in 
non-increasing order, so that 
\[ \lambda_1(T_{\chi,\varphi}) \ge \lambda_2(T_{\chi,\varphi}) \ge 
\lambda_3(T_{\chi,\varphi}) \ge \cdots\,, \]
with each eigenvalue repeated according to its algebraic multiplicity. \\

In the following we shall be interested in the behaviour of the eigenvalue sequence of 
$T_{\chi,k\varphi}$ when $k$ tends to infinity. 
Of fundamental importance for this study is the following result, essentially due to 
Engli\v{s} (see \cite[Theorem~1]{En1}), 
which gives an asymptotic expansion for the Bergman kernels $B_{k\varphi}$ 
as $k$ tends to infinity, provided that $\varphi$ is smooth and strictly psh.

\begin{Theorem} 
\label{theor20}
Let $D$ be a bounded pseudoconvex domain in $\bbC^n$ and $\varphi$ be a 
strictly psh and ${C}^\infty$ function on $D$. Then the following holds. 
\begin{itemize}
\item[(i)]  We have 
\[ k^{-n} B_{k\varphi}(z,z)\,dm_{k\varphi}(z) 
\xrightarrow{k\to \infty} \frac{1}{(2\pi)^n n!} (dd^c \varphi)^n(z)\,. \]
pointwise in $D$, with the left hand side being locally uniformly 
bounded in $D$ for every $k$.
\item[(ii)] We have  
\[ k^{-n} | B_{k\varphi}(z,\zeta)|^2\,dm_{k\varphi}(z)\,dm_{k\varphi}(\zeta) 
\xrightarrow{k\to \infty}\frac{1}{(2\pi)^n n!} (dd^c \varphi)^n|_{z=\zeta} \]
weakly as positive measures on $D\times D$, that is, for every $g \in {C}_c(D\times D)$ 
\begin{multline*} 
k^{-n} \int_D \int_D g(z,\zeta) \abs{B_{k\varphi}(z,\zeta)}^2 
 \,dm_{k\varphi}(z)\,dm_{k\varphi}(\zeta) \\ 
 \xrightarrow{k\to \infty}
\frac{1}{(2\pi)^n n!} \int_D g(z,z)\,(dd^c \varphi)^n (z)\,.
\end{multline*}
\end{itemize}
\end{Theorem}

\begin{Remark}
Assertion (i) is due to Engli\v{s}. Its proof in \cite{En1} 
is based on Fefferman's asymptotic expansion of the Bergman kernel of a 
Forelli-Rudin domain over $D$.

Analogues of assertions (i) and (ii) can be found in the paper \cite{Li} of Lindholm 
(Theorems~10 and 11) for Toeplitz operators on Fock spaces over $\mathbb{C}^n$. 
The method of proving these results is inspired by an approach of Landau~\cite{La}, 
based on studying 
functions concentrated on compact sets, and using 
$L^2$-techniques to obtain precise size estimates of the Bergman kernel both on and off the 
diagonal. Lindholm's proof of the analogue of (ii) in \cite[Theorem 11]{Li} is easily adapted 
to the case of bounded pseudoconvex domains in $\mathbb{C}^n$.
\end{Remark}

Using the previous theorem we are now able to prove a crucial result concerning the 
asymptotics 
of the number of eigenvalues of $T_{\chi,k\varphi}$ greater than a fixed threshold as $k$ tends to 
infinity. 
The proposition and the main idea of its proof are inspired by an analogous result of 
Lindholm \cite[Theorem~13]{Li}. 

\begin{Proposition} \label{cor20} Let $D$ be a bounded pseudoconvex domain in $\bbC^n$ and $
\varphi$ be a strictly psh 
and ${C}^\infty$ function on $D$. Let $\chi\in L^\infty(D)$ be non-zero, non-negative
and have compact support in $D$. Then, for any $\gamma \in (0,1)$, we have   
\[ \lim_{k \to \infty} \sharp \all{m\in \mathbb{N}}{\lambda_m(T_{\chi,k\varphi}) >\gamma 
\|\chi\|_{\infty}}\cdot 
k^{-n} = \frac{1}{(2 \pi)^n n!} \int_D\frac{\chi}{\|\chi\|_{\infty}} \,(dd^c \varphi)^n\,.\]
\end{Proposition}

\begin{proof}
We start by noting that it suffices to show the result under the additional assumption 
that $\|\chi\|_\infty=1$. The general case then follows by observing that 
$\lambda_m(T_{c\chi,\varphi}) = c\lambda_m(T_{\chi,\varphi})$ for any $c\geq 0$ and any 
$m\in \mathbb{N}$.  

Suppose now that $\|\chi\|_\infty=1$. Fix $\gamma\in (0,1)$. We need to show that
\begin{equation}
\label{eq:conclimit}
\lim_{k \to \infty} \sharp \all{m\in \mathbb{N}}{\lambda_m(T_{\chi,k\varphi}) >\gamma}\cdot 
k^{-n} = \frac{1}{(2 \pi)^n n!} \int_D\chi \,(dd^c \varphi)^n\,.
\end{equation}
This will follow from the remarkable fact that both ${\rm Tr}(T_{\chi,k\varphi})$ and  
${\rm Tr}(T^2_{\chi,k\varphi})$ have the same asymptotics as 
$k$ tends to infinity. More precisely, 
by combining Theorem~\ref{theor20} and Lemma~\ref{lem:trace}, we have 
\begin{align}
\label{eq:traceT}
\lim_{k\to \infty} k^{-n}{\rm Tr}(T_{\chi,k\varphi}) & = 
\frac{1}{(2 \pi)^n n!} \int_D\chi \,(dd^c \varphi)^n\,,\\
\label{eq:traceT2}
\lim_{k\to \infty} k^{-n}{\rm Tr}(T^2_{\chi,k\varphi}) & = 
\frac{1}{(2 \pi)^n n!} \int_D\chi \,(dd^c \varphi)^n\,.
\end{align}
In order to simplify notation we shall use the shorthand 
\[ \lambda_m(k)=\lambda_m(T_{\chi,k\varphi})\] 
so that 
${\rm Tr}(T_{\chi,k\varphi})=\sum_m\lambda_m(k)$ and 
${\rm Tr}(T^2_{\chi,k\varphi})=\sum_m\lambda_m^2(k)$. 
Before proceeding we note that by Lemma~\ref{lem:positivity} we have 
$\lambda_m(k)\leq \|\chi\|_\infty=1$ for all $m$ and $k$. Moreover we have 
\[ \lambda_1(k)=\sup\all{(T_{\chi,k\varphi}f,f)}{\|f\|_{H^2_{k\varphi(D)}}=1}\]
\[\geq  \frac{\int_D\chi\cdot \chi_D\,dm_{k\varphi}}{\int_D\chi_D\,dm_{k\varphi }}= 
\frac{\int_D\chi \,dm_{k\varphi}}{m_{k\varphi }(D)}>0\,.\]
Thus, for all $k$ we have 
\[ 0< \sum_m\lambda^2_m(k)\leq \sum_m\lambda_m(k)\,. \]
Combining the above with the equality of the limits (\ref{eq:traceT}) and (\ref{eq:traceT2}) 
we see that for any $\delta >0$ there exists $k_\delta \in \bbN$ such that for any 
$k\ge k_\delta$ we have
\begin{equation} 
\label{eq1}
1-\delta \le \frac{\sum_m\lambda_m^2(k)}{\sum_m \lambda_m(k)}  \le 1\,.
\end{equation}
We shall now establish the following two bounds: 
for every $k\geq k_\delta$ 
\begin{equation} \label{eq2}
\sharp \{m : \lambda_m(k) > \gamma\} \ge \left (1-\frac{\delta}{1-\gamma} \right ) \sum_{m} 
\lambda_m(k)\,,\end{equation}
and for every $k\geq k_\delta$ and every $\gamma'\in (\gamma,1)$ 
\begin{equation} \label{eq3}
\sharp \{m : \lambda_m(k) > \gamma\}
\le \left (\frac{1}{\gamma'} + \frac{\delta}{\gamma(1-\gamma')} \right ) 
\sum_m \lambda_m(k)\,.
\end{equation}
In order to see this fix $k\geq k_\delta$ and define, for every $\beta\in (0,1)$, 
\[ S_{\beta} = \frac{\sum_{m:\lambda_m \le \beta} \lambda_m(k)}{\sum_{m} 
\lambda_m(k)}\,. \]
It now follows from (\ref{eq1}) that
\begin{multline*}
(1-\delta) \sum_m \lambda_m(k) \le \sum_{m:\lambda_m >\gamma} \lambda_m^2(k) + 
\sum_{m:\lambda_m \le \gamma} \lambda_m^2(k) \\
\le \sum_{m:\lambda_m >\gamma} \lambda_m(k) + 
\gamma\cdot\sum_{m:\lambda_m \le \gamma} \lambda_m(k) =
(1-S_{\gamma}) \sum_m \lambda_m(k) + \gamma S_{\gamma} \sum_m \lambda_m(k)\,,
\end{multline*}
so $S_{\gamma} \le \frac{\delta}{1-\gamma}$. Hence
\begin{equation*}
\sharp \{m : \lambda_m(k) > \gamma\} \ge \sum_{m:\lambda_m >\gamma} \lambda_m(k) \ge 
(1-\frac{\delta}{1-\gamma}) \sum_{m} \lambda_m(k)\,,
\end{equation*}
and (\ref{eq2}) is proven. For (\ref{eq3}) we note that for $\gamma < \gamma' <1$ we have  
\begin{align*}
\sharp \{m : \lambda_m(k) > \gamma\}
& = \sharp \{m : \lambda_m(k) > \gamma'\} + 
\sharp \{m : \gamma' \ge \lambda_m(k) > \gamma\} \\
& \le \frac{1}{\gamma'} \sum_{m:\lambda_m >\gamma'} \lambda_m(k)+ \frac{1}{\gamma} 
\sum_{m:\gamma' \ge \lambda_m >\gamma} \lambda_m(k)\\
& \le \frac{1}{\gamma'} \sum_{m} \lambda_m(k)+ \frac{1}{\gamma} S_{\gamma'}\sum_{m} \lambda_m(k)
\end{align*}
and (\ref{eq3}) follows. 

Now, combining (\ref{eq:traceT}) and (\ref{eq2}) gives, for any $\delta>0$
\[ \liminf_{k \to \infty} \sharp \{m : \lambda_m(k) > \gamma\}\cdot k^{-n} \ge
\left (1-\frac{\delta}{1-\gamma} \right )
 \frac{1}{(2\pi)^n n!} \int_D \chi  (dd^c \varphi)^n\,,\]
 which, since $\delta>0$ was arbitrary, yields 
 \begin{equation}
 \label{eq:liminf}
 \liminf_{k \to \infty} \sharp \{m : \lambda_m(k) > \gamma\}\cdot k^{-n} \ge
 \frac{1}{(2\pi)^n n!} \int_D \chi  (dd^c \varphi)^n\,.
 \end{equation}
Similarly, combining (\ref{eq:traceT}) and (\ref{eq3}) gives, for any $\delta>0$ and any 
$\gamma'\in (\gamma,1)$ 
\[ \limsup_{k \to \infty} \sharp \{m : \lambda_m(k) > \gamma\}\cdot k^{-n} \le
\left (\frac{1}{\gamma'}+\frac{\delta}{\gamma(1-\gamma')} \right )
 \frac{1}{(2\pi)^n n!} \int_D \chi  (dd^c \varphi)^n\,;\]
but since $\delta>0$ and $\gamma<\gamma'<1$ were arbitrary, the above implies 
\begin{equation}
\label{eq:limsup}
\limsup_{k \to \infty} \sharp \{m : \lambda_m(k) > \gamma\}\cdot k^{-n} \le
 \frac{1}{(2\pi)^n n!} \int_D \chi  (dd^c \varphi)^n\,.
 \end{equation}
 Combining (\ref{eq:liminf}) and (\ref{eq:limsup}) the limit (\ref{eq:conclimit}) follows.
 \end{proof}

\subsection{Lower bounds for Kolmogorov widths}
\label{sec:2.2}
We shall now use the results of the previous subsection to study the asymptotics of the 
Kolmogorov numbers of certain embedding operators. To be precise, let $D$ be a 
domain in $\mathbb{C}^n$ and let $K$ be a compact subset of $D$ with positive 
$2n$-dimensional Lebesgue measure $m$. 
For $\varphi\in L^\infty(D)$ we write $L^2_\varphi(K)$ 
for the weighted $L^2$-space of Lebesgue measurable functions on $K$ equipped with the 
norm
\[ \|f\|^2_{L^2_\varphi(K)}=\int_K\abs{f}^2e^{-2\varphi}\,dm=\int_K\abs{f}^2\,dm_\varphi\,.\]
Clearly, the natural embedding 
\[ J_{\varphi, K,D}:H^2_\varphi(D)\hookrightarrow L^2_\varphi(K)\] 
given by restricting functions in $H^2_\varphi(D)$ to $K$ is continuous. Moreover, this 
embedding is 
closely related to a  
Toeplitz operator on $H^2_\varphi(D)$, provided that its symbol is the indicator function 
of $K$, as the following lemma shows. 

\begin{Lemma}
\label{lem:TisJstarJ}
Let $D\subset \mathbb{C}^n$ be a domain, $K$ a compact subset of $D$ 
with positive Lebesgue measure and 
$\varphi\in L^\infty(D)$. If $\chi$ is the indicator function of $K$, then
\[ T_{\chi,\varphi}= J^*_{\varphi, K,D}J_{\varphi,K,D}\,. \]
In particular, we have 
\[ \lambda_m(T_{\chi,\varphi})=d_m(J_{\varphi, K,D})^2 \quad (\forall m \in \mathbb{N})\,.\]  
\end{Lemma}

\begin{proof}
Let $f,g\in H^2_\varphi(D)$. Then 
\begin{multline*}
(T_{\chi,\varphi}f,g)_{H^2_\varphi(D)}=(P_\varphi M_\chi J_\varphi f,g)_{H^2_\varphi(D)} \\
= (M_\chi J_\varphi f, J_\varphi g)_{L^2_\varphi(D)} 
= \int_D\chi f \overline{g} \,dm_\varphi = 
\int_K f\overline{g}\,dm_\varphi 
\end{multline*}
and 
\[ (J^*_{\varphi,K,D}J_{\varphi,K,D}f,g)_{H^2_\varphi(D)}=
(J_{\varphi,K,D} f,J_{\varphi,K,D}g)_{L^2_\varphi(K)} =\int_K f\overline{g}\,dm_\varphi \,,\]
so $T_{\chi,\varphi}=J^*_{\varphi,K,D}J_{\varphi,K,D}$ as claimed.

Recalling that, by definition, 
the $m$-th singular value of the 
compact operator $J_{\varphi,K,D}$ is equal to  
$\lambda_m(J^*_{\varphi, K,D}J_{\varphi,K,D})^{1/2}$,   
the remaining assertion now follows from the fact that for compact operators on a Hilbert space
the singular values and Kolmogorov numbers coincide 
(see, for example, \cite[Theorem~2.11.9]{pietsch87}).
\end{proof}

In order to apply the results of the previous subsection we also
require the following simple relationship between the Kolmogorov 
numbers of $J_{k\varphi, K,D}$ and $J_{0,K,D}$, that is, the natural embedding of the standard, 
unweighted Bergman space $H^2(D)$ in the unweighted Lebesgue space $L^2(K)$. 

\begin{Lemma}
\label{lem:dmkphi}
Let $D\subset \mathbb{C}^n$ be a domain, $K$ a compact subset of $D$ 
with $m(K)>0$ and $\varphi\in L^\infty(D)$. If $\varphi\leq 0$ on $D$ and 
$\varphi \geq -1$ on $K$, then
\[ d_m(J_{k\varphi, K,D})\leq e^k d_m(J_{0, K,D}) 
\quad(\forall k,m\in \mathbb{N})\,. \]
\end{Lemma}

\begin{proof}
Let $I_{k,K}:L^2_0(K)\to L^2_{k\varphi}(K)$ and 
$I_{k,D}:H^2_0(D)\to H^2_{k\varphi}(D)$ denote the respective identities. Then
\[ J_{k\varphi,D,K}=I_{k,K}J_{0,K,D}I_{k,D}^{-1}\,,\]
and so 
\[ d_m(J_{k\varphi,D,K})\leq \|I_{k,K}\|d_m(J_{0, K,D})\|I_{k,D}^{-1}\|\,.\]
Note now that 
$\|I_{k,K}\|\leq e^k$, since, using $\varphi \geq -1$ on $K$, we have  
\[ \|I_{k,K}f\|^2_{L^2_{k\varphi}(K)}=\int_K\abs{f}^2e^{-2k\varphi}\,dm \leq 
e^{2k}\int_K \abs{f}^2\,dm=
e^{2k} \|f\|^2_{L^2_0(K)}\,. \]
Furthermore, $\|I^{-1}_{k,D}\|\leq 1$, since, 
using $\varphi \leq 0$ on $D$, we have  
\[ \|I^{-1}_{k,D}f\|^2_{H^2_{0}(D)}=\int_D\abs{f}^2\,dm \leq 
\int_D\abs{f}^2e^{-2k\varphi}\,dm 
=\|f\|^2_{H^2_{k\varphi}(D)}\,, \]
and the assertion follows. 
\end{proof}

Combining the previous two lemmas with Proposition~\ref{cor20} we obtain the following result. 

\begin{Proposition}
\label{prop:dmlowersmoothphi}
Let $D$ be a bounded pseudoconvex domain in $\mathbb{C}^n$ and 
$K$ a compact subset of $D$ 
with $m(K)>0$. If $\varphi$ is strictly psh 
and $C^\infty$ on $D$ with 
$\varphi \leq 0$ on $D$ and $\varphi \geq -1$ on $K$, then  
\[ \limsup_{m\to \infty} \frac{-\log d_m(J_{0,K,D})}{m^{1/n}} 
\leq 2\pi \left ( \frac{n!}{\int_K\,(dd^c\varphi)^n}\right )^{1/n}\,.\]
\end{Proposition}

\begin{proof}
We start by observing that $\int_K\,(dd^c\varphi)^n>0$. 
Let now $\chi$ be the indicator function of $K$ and fix $\gamma\in (0,1)$. 
Using Lemma~\ref{lem:TisJstarJ} and Lemma~\ref{lem:dmkphi} it follows that 
\[ e^{2k}d_m(J_{0,K,D})^2 \geq  d_m(J_{k\varphi,K,D})^2=\lambda_m(T_{\chi,k\varphi})
\quad (\forall k,m\in \mathbb{N})\,,
\]
hence 
\[ \sharp \all{m\in \mathbb{N}}{d_m(J_{0,K,D})>\gamma e^{-k}}
\geq \sharp \all{m\in \mathbb{N}}{\lambda_m(T_{\chi,k\varphi})>\gamma^2} \,.\]  
Applying Proposition~\ref{cor20} to the above we see that 
\begin{equation}
\label{eq:liminf0}
\liminf_{k\to \infty} \sharp \all{m\in \mathbb{N}}{d_m(J_{0,K,D})>\gamma e^{-k}}\cdot k^{-n} \geq C\,,
\end{equation}
where
\[ C= \frac{1}{(2\pi)^n n!} \int_K\,(dd^c\varphi)^n\,.\]
Write $d_m$ for $d_m(J_{0,K,D})$ and fix $C'\in (0,C)$. 
Using (\ref{eq:liminf0}) it follows that there is $k'\in \mathbb{N}$ such that 
\[ \sharp \all{m\in \mathbb{N}}{d_m>\gamma e^{-k}}\geq C'k^n 
\quad (\forall k\geq k')\,, \]
so 
\begin{equation}
\label{eq:dfloor}
 d_{\floor{k^nC'}}\geq \gamma e^{-k} \quad (\forall k\geq k')\,,
 \end{equation}
where $\floor{\cdot}$ denotes the floor function. 
After possibly enlarging $k'$ we shall 
assume that $(k')^n C'\geq 1$. Now, for any $m\geq \floor{(k')^n C'}$ choose 
$k_m \ge k'$ to be the largest $k\in \mathbb{N}$ such that 
\[ \floor{k^nC'}\leq m < \floor{(k+1)^nC'}\,.\]
Then 
$\lim_{m\to \infty}k_m=\infty$ and we have, for every $m\geq \floor{(k')^nC'}$, 
\[ d_m\geq d_{\floor{(k_m+1)^nC'}} \geq \gamma e^{-(k_m+1)}\,, \]
which follows from (\ref{eq:dfloor}) together with the observation
that $k_m\geq k'$, and 
\[ m\geq \floor{k_m^nC'}>k_m^nC'-1\,, \]
and so 
\[ \frac{-\log d_m}{m^{1/n}} \leq \frac{k_m+1-\log\gamma}{(k_m^nC'-1)^{1/n}}\,.\]
This implies
\[ \limsup_{m\to \infty }\frac{-\log d_m}{m^{1/n}} \leq \frac{1}{(C')^{1/n}}\,.\]
But since $C'<C$ was arbitrary, the assertion follows. 
\end{proof}

\subsection{Sharp lower bounds for Kolmogorov widths with respect to $L^2$-topologies}
\label{sec:2.3}
The bounds for the Kolmogorov numbers for the natural embedding
$H^2(D)\hookrightarrow L^2(K)$ for a bounded pseudoconvex domain $D$ in $\mathbb{C}^n$ 
and $K$ a compact subset of $D$ with positive Lebesgue measure were obtained under the 
additional 
assumption that $\varphi$ be strictly psh and smooth on $D$. 
In order to make the connection with the relative capacity $C(K,D)$ we want to choose $\varphi$ in 
Proposition~\ref{prop:dmlowersmoothphi} to be the 
upper semicontinuous regularisation $u_{K,D}^*$ of the 
relative extremal function $u_{K,D}$. 
However, even if $D$ is bounded and hyperconvex, and $K$ is regular in $D$, 
then $u_{K,D}=u_{K,D}^*$ is, in general, merely continuous on the closure $\clo{D}$ and not necessarily smooth there. 
Nevertheless, using results of Cegrell \cite{Ce1}, we will be able to approximate $u_{K,D}^*$ 
by smooth functions with the desired properties. 

In order to state these results we need some more notation. 
Let $PSH^-(D)$ denote the set of non-positive 
psh functions on $D$ and 
${\cal E}_0(D)$ the class of bounded psh functions $\psi$ 
such that $\lim_{z\to \xi} \psi(z) = 0$ for all $\xi \in \partial D$ and $\int_ D (dd^c \psi)^n < +\infty$. 
Note that if a function $\varphi \in {\cal E}_0(D)$ is continuous in $D$, then 
$\varphi$ is actually continuous up to the boundary.

\begin{Theorem}[\cite{Ce1}] 
\label{theor21}
Let $D$ be a hyperconvex domain in $\bbC^n$. For 
every $u \in PSH^-(D)$, 
there is a decreasing sequence $(\psi_j)_{j\in \mathbb{N}}$ 
of functions in ${\cal E}_0 \cap {C}^\infty(D)$ such that $\psi_j \to u$ as $j \to \infty$, pointwise in $D$.
\end{Theorem}

\begin{Corollary}[\cite{Ce1}]
\label{cor21}
Let $D$ be a bounded hyperconvex domain in $\bbC^n$. Then there is a strictly psh exhaustion function 
$\psi \in {\cal E}_0 \cap {C}^\infty(D)$ for $D$.
\end{Corollary}

Recall that for any compact subset $K$ in an open set $D$, the relative capacity $C(K,D)$ is 
defined as follows
\[ C(K,D)=\sup \all{\int_K\,(dd^c\varphi)^n}{\varphi \in PSH(D,(-1,0))}\,.\]

Using the previous results we obtain the following characterisation 
of the relative capacity $C(K,D)$ in the case where $D$ is a bounded hyperconvex domain. 

\begin{Corollary}
\label{cor:regsupCKD}
Let $D$ be a bounded hyperconvex domain in $\mathbb{C}^n$ containing a compact subset $K$. 
Then 
\[ C(K,D)=\sup \all{\int_K\,(dd^c\varphi)^n}{\varphi\in SPSH(D,(-1,0))\cap C^\infty(D)}\,.\]
Here, $SPSH(D,(-1,0))$ denotes the collection of all strictly psh functions on $D$ with values in 
$(-1,0)$. 
\end{Corollary}

\begin{proof}
First, we approximate $u_{K,D}^*$ by a decreasing sequence $(\psi_j)_{j\in \mathbb{N}}$ 
of functions in ${\cal E}_0 \cap {C}^\infty(D)$ such that $\psi_j \to u_{K,D}^*$ pointwise on $D$ as 
$j \to \infty$ (see Theorem~\ref{theor21}). 
Since $u_{K,D}^* \ge -1$ in $D$, each function $\psi_j$ satisfies $\psi_j \ge -1 $ in $D$.

Secondly, let $\psi$ be the function of Corollary~\ref{cor21} and replace 
each function $\psi_j$ by the function 
\[ \varphi_{j,\epsilon} = \frac{\psi_j + \epsilon \psi}{1+\epsilon M}\,,\]
where $\epsilon >0$ is small and $M>0$ is chosen so that $-M\le \psi\le 0$ on $\clo{D}$, which is 
possible since $\psi$ is bounded on  $\clo{D}$. Each $\varphi_{j,\epsilon}$ 
is a strictly psh exhaustion function in ${\cal E}_0 \cap {C}^\infty(D)$ 
for $D$ and $\varphi_{j,\epsilon} \ge -1$ in $D$.

Since, for fixed $\epsilon$, the sequence $(\varphi_{j,\epsilon})_{j\in \mathbb{N}}$ 
decreases in $D$  
to $\frac{u_{K,D}^* + \epsilon \psi}{1+\epsilon M}$, which is bounded in $D$, we have 
\[ \lim_{j\to \infty} (dd^c \varphi_{j,\epsilon})^n= \,\left(dd^c \left(\frac{u_{K,D}^* + 
\epsilon \psi}{1+\epsilon M}\right)\right)^n\,,\]
in the sense of weak convergence of measures, and in particular 
\[ \lim_{j\to \infty}\int_K\,(dd^c \varphi_{j,\epsilon})^n= 
\int_K\,\left(dd^c \left(\frac{u_{K,D}^* + \epsilon \psi}{1+\epsilon M}\right)\right)^n\,.\]

Moreover, since $\frac{u_{K,D}^* + \epsilon \psi}{1+\epsilon M}$ 
converges uniformly to $u_{K,D}^*$ on $\clo{D}$ when $\epsilon \to 0$, we have 
\[ \lim_{\epsilon\to 0}\int_K\,\left(dd^c\left(\frac{u_{K,D}^* + \epsilon \psi}{1+\epsilon M}\right)
\right)^n 
=\int_K (dd^c u_{K,D}^*)^n=C(K,D)\,.\]
\end{proof}

Combining the results of this and the previous subsections we now obtain the following sharp 
lower bound 
for the Kolmogorov numbers of the natural embedding $H^2(D)\hookrightarrow L^2(K)$. 

\begin{Theorem}
\label{thm:sharpL2lowerbounds}
Let $D$ be a bounded hyperconvex domain in $\mathbb{C}^n$ containing a compact subset $K$ 
with positive 
Lebesgue measure. Then 
\[ \limsup_{m\to \infty} \frac{-\log d_m(H^2(D)\hookrightarrow L^2(K))}{m^{1/n}} 
\leq 2\pi \left ( \frac{n!}
{C(K,D)} \right )^{1/n} \,.\] 
\end{Theorem}
\begin{proof}
Follows from Proposition~\ref{prop:dmlowersmoothphi} 
and Corollary~\ref{cor:regsupCKD}.
\end{proof}

\subsection{From $L^2$ to $L^\infty$}
\label{sec:2.4}
So far we have obtained lower bounds for the Kolmogorov widths of holomorphic functions in spaces carrying $L^2$-topologies. We shall now
consider bounds with respect to spaces arising from supremum norm topologies, as in the original 
formulation of Kolmogorov's problem. We shall achieve this by using 
simple bounds for the Kolmogorov numbers coupled with approximation arguments.

We start with the following simple observation which allows us to change the topology on the 
target space. 

\begin{Lemma}
\label{lem:H2DCK}
Let $D$ be a domain in $\mathbb{C}^n$ and $K$ a compact subset of $D$ 
with positive Lebesgue measure. Then 
\[ d_m(H^2(D)\hookrightarrow A(K))
\geq \frac{1}{\sqrt{m(K)}}d_m(H^2(D) \hookrightarrow 
L^2(K))\quad (\forall m\in \mathbb{N})\,.\]
\end{Lemma}

\begin{proof}
Let
\begin{align*}
J_1&:H^2(D)\hookrightarrow A(K)\,,\\
J_2&:A(K)\hookrightarrow L^2(K)\,,\\
J&:H^2(D)\hookrightarrow L^2(K)\,,
\end{align*}
denote the natural embeddings. Then $J=J_2J_1$. Clearly, $J_2$ is
bounded with $\|J_2\|\leq \sqrt{m(K)}$, 
so 
\[ d_m(J)=d_m(J_2J_1)\leq \|J_2\|d_m(J_1) \quad (\forall m\in
\mathbb{N})\,,\]
and the assertion follows.  
\end{proof}

In order to change the topology on the original space we use a similar argument. 

\begin{Lemma}
\label{lem:H2DpCK}
Let $D$ and $D'$ be domains in $\mathbb{C}^n$ with $D\subset\subset D'$ 
and let $K$ be a compact subset of $D$.  
Then there is a constant $C>0$, depending on $D$ and $D'$ only, and not on $m$, such that 
\[ d_m(H^\infty(D)\hookrightarrow A(K))
\geq C d_m(H^2(D')\hookrightarrow A(K))
\quad (\forall m\in \mathbb{N})\,.\]
\end{Lemma}

\begin{proof}
Let
\begin{align*}
J_1&:H^2(D')\hookrightarrow H^\infty(D)\,,\\
J_2&:H^\infty(D)\hookrightarrow A(K)\,,\\
J&:H^2(D')\hookrightarrow A(K)\,,
\end{align*}
denote the natural embeddings. Then $J=J_2J_1$. Now $J_1$ is bounded
by a standard result for Bergman spaces (see, for example,
\cite[Lemma 1.4.1]{krantzbook82}) and clearly non-zero, 
so 
\[ d_m(J)=d_m(J_2J_1)\leq d_m(J_2)\|J_1\| \quad (\forall m\in \mathbb{N})\,,\]
and the assertion follows with $C=\|J_1\|^{-1}$. 
\end{proof}

Before stating and proving the main result of this subsection, we require one more 
result, a convergence lemma for strictly hyperconvex domains $D$. 
For such domains, using the same notation as in Subection~\ref{sec:1.1}, 
it follows that there exists a bounded domain $\Omega$ and an exhaustion function 
$\varrho \in PSH( \Omega, (-\infty,1))\cap C(\Omega)$ 
such that $D=\all{z\in\Omega}{\varrho (z)<0}$ and for all real numbers $c \in [0,1]$, the 
open set $\all{z \in \Omega}{\varrho (z)<c}$ is connected.
Given any integer $j \ge 1$, we now define 
\begin{equation}
\label{eq:Djdef}
D_j=\all{z\in\Omega}{\varrho (z)<1/j}\,.
\end{equation} 
This decreasing sequence of bounded hyperconvex domains $(D_j)_j$ satisfies the following 
lemma which is a generalisation of Lemma $2.6$ in \cite{N4}. 

\begin{Lemma}
\label{lem:10}
Let $D$ be a strictly hyperconvex domain in $\mathbb{C}^n$ and 
let $K$ be a compact subset of $D$. If $(D_j)_j$  denotes the sequence of bounded hyperconvex domains defined in (\ref{eq:Djdef}), 
then the increasing sequence $(u_{K,D_j}^*)_j$ 
converges quasi-everywhere in $D$ to $u_{K,D}^*$ and the 
increasing sequence of capacities $(C(K,D_j))_j$ converges to the capacity $C(K,D)$. 
\end{Lemma}

\begin{proof}
First we note that since $K \subset D\subset D_j$, we have 
$-1\le u_{K,D_j} \le u_{K,D}$ and $u_{K,D_j}^* \le u_{K,D}^*\le 0$ in $D$ as well as 
$C(K,D_j) \le C(K,D)$ (see \cite{B-T2} or \cite[p.~120]{K2}). 

We now collect some simple properties satisfied by the functions $u_{K,D_j}$ and $u_{K,D}$.
There exists a positive constant $c$ sufficiently large such that $c\varrho \le -1$ in $K$ and 
$c(\varrho -1/j) \le u_{K,D_j}$ in $D_j$, for any $j$. 
The sequence $(u_{K,D_j}^*)_j$ is increasing in $D$, since the sequence $(D_j)_j$ is decreasing. 
Let $v$ be the function defined on $D$ by  
\[ v:=\lim_{j\to \infty} u_{K,D_j}^*\,.\]
Now $v^*$ is psh in $D$. Moreover,   
$v$ and $v^*$ are equal quasi-everywhere in $D$ (that is, 
they are equal except perhaps for a pluripolar set, see \cite{B-T2}). 
As $u_{K,D_j}^* \in PSH(D,(-1,0))$, so $v^* \in PSH(D,(-1,0))$.

Since each $u_{K,D_j}^*$ satisfies $u_{K,D_j}^* \le u_{K,D}^*$ in $D$, it follows that 
$v^*$ also satisfies $v^* \le u_{K,D}^*$ in $D$.

Moreover, the sequence of positive measures $((dd^c u_{K,D_j}^*)^n)_j$ converges to the positive 
measure $(dd^c v^*)^n$ 
in the weak$^*$-topology (see \cite{B-T2} or \cite[p.~125]{K2}). 
In particular, since for any $j$ $(dd^c u_{K,D_j}^*)^n=0$ in $D_j\setminus K$, 
we have $(dd^c v^*)^n=0$ in $D\setminus K$.

Note that $v^*(w)$ converges to $0$ when $w \to z$, for any $z \in \partial D$. 
Indeed, for any $j$ we have  
\[ -c/j \le \liminf_{w \to z} u_{K,D_j}^*(w) \le \liminf_{w \to z} v^*(w)\le \limsup_{w \to z} v^*(w) \le 0\,.\]
Now recall that $C(K,D)=\int_K (dd^c u_{K,D}^*)^n = \int_D (dd^c u_{K,D}^*)^n$ and that 
the same equalities are satisfied for $C(K,D_j)$ and $u_{K,D_j}^*$. 
Thus, using 
the Comparison Theorem of Bedford and Taylor (see \cite{B-T2} or \cite[p.~126]{K2}) 
we deduce that   
\begin{multline*}
C(K,D)=\int_D (dd^c u_{K,D}^*)^n \le \int_D (dd^c v^*)^n = \int_K (dd^c v^*)^n \\
=\lim_{j\to \infty} \int_K (dd^c u_{K,D_j}^*)^n=\lim_{j\to \infty} C(K,D_j) \le C(K,D)\,,
\end{multline*}
and that $v^*$ is identical to $u_{K,D}^*$ in $D$. Since $v=v^*$ quasi-everywhere in 
$D$, the 
proof is complete.
\end{proof}

The following is the main result of this subsection, 
a sharp lower bound for the 
asymptotics of the Kolmogorov numbers  
$d_m(H^\infty(D)\hookrightarrow A(K))$ under the hypotheses of 
Theorem~\ref{thm:sharpL2lowerbounds}, except that we require a bit
more regularity for $D$, namely that $D$ be not just bounded and hyperconvex but strictly 
hyperconvex.

\begin{Theorem}
\label{theor11}
Let $D$ be a strictly hyperconvex domain in $\mathbb{C}^n$ containing a compact subset $K$ with positive 
Lebesgue measure. Then 
\[ \limsup_{m\to \infty} \frac{-\log d_m(H^\infty(D)\hookrightarrow A(K))}{m^{1/n}} 
\leq 2\pi \left ( \frac{n!}
{C(K,D)} \right )^{1/n} \,.\] 
\end{Theorem}

\begin{proof}
Combining Lemma~\ref{lem:H2DCK} and Lemma~\ref{lem:H2DpCK}, applied to 
$K \subset D \subset D_j$ with $D_j$ defined in (\ref{eq:Djdef}),  
we obtain for any $j$
\begin{multline*}
d_m(H^\infty(D)\hookrightarrow A(K)) \ge C_j d_m(H^2(D_j)\hookrightarrow A(K))\\
\ge \frac{C_j}{\sqrt{m(K)}} d_m(H^2(D_j)\hookrightarrow L^2(K))\,,
\end{multline*}
where $C_j=|| J_{H^2(D_j) \hookrightarrow H^\infty(D)}||^{-1}$.

Applying Theorem~\ref{thm:sharpL2lowerbounds} to bound 
$d_m(H^2(D_j)\hookrightarrow L^2(K))$ from below, we deduce for any $j$  
\begin{equation*}
\limsup_{m\to \infty} \frac{-\log d_m(H^\infty(D)\hookrightarrow A(K))}{m^{1/n}} 
\leq 2\pi \left ( \frac{n!}{C(K,D_j)} \right )^{1/n}\,,
\end{equation*}
and the assertion now follows from Lemma~\ref{lem:10}. 
\end{proof}

\subsection{Generalisation to non-pluripolar $K$}
\label{sec:2.5}
In the previous subsection we have obtained lower bounds for the Kolmogorov 
numbers of the embedding $J:H^\infty(D)\hookrightarrow A(K)$ for $D$ strictly 
hyperconvex and $K\subset D$ compact with positive Lebesgue measure. 
We shall now explain how to obtain Theorem~\ref{theor11} in the more general case 
where the compact set $K$ is only assumed to be non-pluripolar, thus finishing the proof 
of the first half our main theorem (Theorem~\ref{theor10}). We shall achieve this by
choosing suitable external approximations of the holomorphically convex 
hull of $K$ in $D$.

We start with the following simple observation. 
If $D$ is hyperconvex (but not necessarily bounded) and $K$ is 
a non-pluripolar compact subset of $D$, which 
we do not assume to be holomorphically convex in $D$, then the 
relative extremal function $u_{K,D}$ is lower semicontinuous on $D$, see 
\cite[Corollary~4.5.11]{K2}. 
Thus, 
the upper level sets $\all{z \in D}{u_{K,D}(z) >-1+c}$ are open in $D$, 
for any real number $0 \le c< 1$. Since, 
for any $w\in \partial D$, the relative extremal function 
$u_{K,D}(z)$ tends to $0$ as 
$z$ tends to $w$, it follows that, for any real number $0 \le c <1$, 
\begin{equation}
\label{eq:Kcdef}
K_c=\all{z \in D}{u_{K,D}(z) \le -1+c}
\end{equation}
is a compact subset of $D$, which, for $0<c<1$, 
has positive Lebesgue measure, since it contains the 
open set $\all{z\in D}{u_{K,D}^*<-1+c}$.

We also note that 
\[  \bigcap_{0 < c< 1} K_c = \hat{K}_D\,,\]
where $\hat{K}_D$ is the holomorphically convex hull of $K$ in the open 
hyperconvex set $D$. Moreover, $u_{K,D}=u_{\hat{K}_D,D}$ and 
$C(K,D)=C(\hat{K}_D,D)$.

\begin{Lemma}
\label{lem:11}
Let $D$ be a hyperconvex domain in $\mathbb{C}^n$ and let 
$K$ be a non-pluripolar compact subset of $D$.  
Then, for any real number $0<c<1$, we have 
\[ \max\{\frac{u^*_{K,D}}{1-c}, -1\} = u^*_{K_c,D} \text{  on  } D \]
and 
\[ C(K_c,D)=\frac{C(K,D)}{(1-c)^{n}}\,,\]
where $K_c$ is given by (\ref{eq:Kcdef}). 
\end{Lemma}

\begin{proof}
We start by observing that the 
relative extremal function $u_{K_c,D}$ for $K_c $ in $D$,  satisfies 
\begin{equation}
\label{eq:lem11:eq1}
 \max\{\frac{u_{K,D}}{1-c}, -1\} \le u_{K_c,D} \text{ on } D\,.
 \end{equation}
Indeed, any psh function $w$ appearing in the definition of 
$u_{K,D}$ satisfies $w \le u_{K,D}$ on $D$, so $w \le -1+c$ on $K_c$. 
Thus $\max\{\frac{w}{1-c}, -1\}$ is a negative psh function on $D$ which is less than or 
equal to $-1$ on $K_c$, from which, using the definition of $u_{K_c,D}$, 
we conclude that $\max\{\frac{w}{1-c}, -1\} \le u_{K_c,D}$ on $D$, 
and (\ref{eq:lem11:eq1}) follows.

Let $v_c$ denote the following psh function on $D$
\[ v_c=\max\{\frac{u^*_{K,D}}{1-c}, -1\}\,.\]
Now, $v_c$ is a negative psh function on $D$, which is greater or equal to $-1$ on  $D\setminus K_c$. Indeed, on $D\setminus K_c$, we have 
$u_{K,D} >-1+c$ and thus $u^*_{K,D} >-1+c$ as well. 
In addition, $v_c$  tends to $0$ as $z$ tends to $w$, 
for any $w  \in \partial D$. 

Since $(dd^c u^*_{K,D})^n=0$ on $D \setminus K$, 
we also have $(dd^c v_c)^n =0$ on $D\setminus K_c$. 
By the maximality of the function $v_c$, we deduce that 
$v_c \ge w$ in $D \setminus K_c$, for any psh function $w$ which appears in the 
definition of $u_{K_c,D}$. 
Thus $v_c \ge u_{K_c,D}$ and even $v_c \ge u^*_{K_c,D}$ on $D \setminus K_c$, 
since $v_c$ is psh.

Combining the bounds in the previous two paragraphs we deduce that
\[ v_c = u^*_{K_c,D} \mbox{ on } D \setminus K_c\,.\]
Summarising, we have 
\[ \max\{\frac{u^*_{K,D}}{1-c}, -1\} = u^*_{K_c,D} \text{ on } D \setminus K_c \]
and 
\[ \max\{\frac{u_{K,D}}{1-c}, -1\} = u_{K_c,D}=-1 \text{ on } K_c\,.\]
We know that $u_{K,D} = u^*_{K,D}$ and $u_{K_c,D} = u^*_{K_c,D}$ except for a 
pluripolar set in $D$. As two psh functions which are equal except for a pluripolar set 
are in fact equal everywhere, we finally deduce that 
\[ \max\{\frac{u^*_{K,D}}{1-c}, -1\} = u^*_{K_c,D} \text{ on } D\,,\]
and the proof is complete.
\end{proof}

We are finally able to prove the main result of this section, the first half of our main theorem. 
\begin{Theorem}
\label{theor12}
Let $D$ be a strictly hyperconvex domain in $\mathbb{C}^n$ and $K$ a non-pluripolar compact subset of $D$. Then 
\[ \limsup_{m\to \infty} \frac{-\log d_m(H^\infty(D)\hookrightarrow A(K))}{m^{1/n}} 
\leq 2\pi \left ( \frac{n!}
{C(K,D)} \right )^{1/n} \,.\] 
\end{Theorem}

\begin{proof}
Fix $c$ with $0<c<1$. Let $K_c$ be as in (\ref{eq:Kcdef}) and define the canonical 
embeddings
\[ J:H^\infty(D)\hookrightarrow A(K)\,, \]
\[ J_c:H^\infty(D)\hookrightarrow A(K_c)\,. \]
We start by observing that the Two Constants Theorem (see, for example, 
\cite[Proposition~4.5.6]{K2}) implies that for any $f\in H^\infty(D)$ with 
$\norm{f}_{H^\infty(D)}\leq 1$ we have 
\[ \norm{f}_{A(K_c)}\leq 
\norm{f}_{A(K)}^{1-c}
\,,\]
which implies the following relation between the Gelfand numbers of the embeddings 
$J$ and $J_c$
\[ c_m(J_c)\leq c_m(J)^{1-c} \quad (m \in \mathbb{N})\,.\]
Thus 
\[ 
\limsup_{m\to \infty} \frac{-\log c_m(J)}{m^{1/n}} \leq 
\limsup_{m\to \infty} \frac{1}{1-c}\frac{-\log c_m(J_c)}{m^{1/n}}\,,
\]
and so, using (\ref{eq:dmam}) and (\ref{eq:cmam}), it follows that 
\[ 
\limsup_{m\to \infty} \frac{-\log d_m(J)}{m^{1/n}} \leq 
\limsup_{m\to \infty} \frac{1}{1-c}\frac{-\log d_m(J_c)}{m^{1/n}}\,. 
\]
Now, since $K_c$ has positive Lebesgue measure, Theorem~\ref{theor11} implies 
\[ 
\limsup_{m\to \infty} \frac{1}{1-c}\frac{-\log d_m(J_c)}{m^{1/n}}
\leq \frac{2\pi }{1-c} \left ( \frac{n!}{C(K_c,D)}\right )^{1/n}\,, 
\]
while Lemma~\ref{lem:11} gives 
\[ \frac{2\pi }{1-c} \left ( \frac{n!}{C(K_c,D)}\right )^{1/n}
= 2\pi \left ( \frac{n!}{C(K,D)}\right )^{1/n}\,. \] 
Thus, all in all, we have 
\[ \limsup_{m\to \infty} \frac{-\log d_m(J)}{m^{1/n}} 
\leq 2\pi \left ( \frac{n!}
{C(K,D)} \right )^{1/n} \,,\]
and the proof is finished.  
\end{proof}

\section{Upper bound for the Kolmogorov widths}
\label{sec:3}

Our strategy for determining sharp upper bounds for the Kolmogorov widths  
is based on the Bergman-Weil formula coupled 
with an approximation argument. More precisely, 
we shall start by letting $K$ and $D$ be special holomorphic polyhedra in $\mathbb{C}^n$ (a notion we 
shall recall below), 
in which case sharp upper bounds for the Kolmogorov numbers
$d_n(H^\infty(D)\hookrightarrow A(K))$ 
can be obtained using the Bergman-Weil integral formula. 
Next, for any pair $(K,D)$, where $K$ is a regular compact subset of a strictly 
hyperconvex domain $D$, we will simultaneously approximate $\hat{K}_D$ externally 
and $D$ internally by two special holomorphic polyhedra defined by the same holomorphic 
mapping in such a way that 
the relative capacity of the approximations will converge to the relative capacity $C(K,D)$. 
This will allow us to deduce a sharp upper bound for 
the Kolmogorov widths in the case where 
$K$ is a regular compact subset in 
a strictly hyperconvex domain $D$. From this, we will deduce an upper bound for the 
Kolmogorov widths of a general pair $(K,D)$, also termed a \emph{condenser} in this section, 
where $K$ is any compact subset of a bounded 
and hyperconvex domain $D$ in $\mathbb{C}^n$.

\subsection{Upper bounds in case $K$ and $D$ are special holomorphic polyhedra}

We start by recalling the notions of holomorphic polyhedron and special 
holomorphic polyhedron. 
Let $\Omega$ be an open subset of $\mathbb{C}^n$ and $N$ a positive integer. 
A \emph{holomorphic polyhedron} of type $N$ in $\Omega$ 
is a finite union of relatively compact connected
components of the subset of $\Omega$ of the form
\[ \all{z \in \Omega}{|f_j(z)| < 1 \mbox{ for  all } j\in \{1,\ldots ,N\}}\,,\]
where each $f_j:\Omega\to \mathbb{C}$ is holomorphic.

Clearly, a holomorphic polyhedron of type $N$ is also of type $N+1$, so the minimal
type of a given polyhedron is of particular interest. Note that if 
$\Omega$ is a holomorphically 
convex domain in $\mathbb{C}^n$ and $\mathcal{P}$ a holomorphic
polyhedron of type $N$ in $\Omega$ then $N \ge n$.
Thus, in this case there is a nontrivial lower bound for the type of a holomorphic 
polyhedron,
and polyhedra of this minimal type play a special role: 
 a 
holomorphic polyhedron of type $n$ in a holomorphically convex domain $\Omega$ in 
$\mathbb{C}^n$ is called a {\it special holomorphic polyhedron} (see Bishop \cite{Bi}).

For the rest of this subsection we shall focus on condensers of the 
form $(K,D)=(\clo{{\cal U}_b}, {\cal U}_a)$, where both ${\cal U}_a$
and ${\cal U}_b$ are special holomorphic polyhedra  obtained as follows. 

Let $F= (f_1,$ $\ldots, f_n) :\Omega \to \bbC^n$ denote a holomorphic mapping on a 
pseudoconvex open set $\Omega$ in $\bbC^n$. 
For $a\in (0,\infty)^n$ write 
\[ {\cal P}_a=\all{z\in \Omega}{|f_k(z)| <a_k \mbox{ for all } k\in \{1,\ldots,n\}}, \]
and assume that the open set ${\cal P}_a$ has relatively compact connected components. 
We now define the special holomorphic polyhedron ${\cal U}_a$ to be 
a finite union of such components. 

Given $b\in (0,\infty)^n$ with $b_k< a_k$ for $k\in \{1, \ldots, n\}$ we define a second special 
holomorphic polyhedron ${\cal U}_b$ as the finite union of 
connected components of the open set 
\[ {\cal P}_b=\all{ z\in \Omega}{ | f_k(z)| <b_k \mbox{ for all } k\in \{1,\ldots ,n\}},\] 
which belong to ${\cal U}_a$. 

In this case we shall call the condenser $(\clo{{\cal U}_b}, {\cal U}_a)$ a 
\emph{special holomorphic polyhedral condenser} and $F$ the  
\emph{underlying mapping}. 

Note that the underlying  mapping $F$ is finite and proper from ${\cal U}_a$ onto the polydisc 
$P(O,a):=\all{z \in \bbC^n}{ |z_j | < a_j, \, 1\le j \le n}$.  Moreover, $F|_{{\cal U}_a}$ is an 
unramified covering over the open set $P(O,a)$. We shall refer to 
the number of sheets of this covering as the \emph{multiplicity} of $F$. 

It turns out that for special holomorphic polyhedral condensers there is a simple formula for the 
corresponding relative capacity.

\begin{Proposition} \label{theor30}
Let $(\clo{{\cal U}_b},{\cal U}_a)$ be a special holomorphic polyhedral condenser 
with underlying mapping $F=(f_1, \ldots, f_n)$ in $\mathbb{C}^n$. 
Then the corresponding relative extremal function is given explicitly by 
\[ u_{\clo{{\cal U}_b},{\cal U}_a}(z)=
u_{\clo{P(O,b)}, P(O,a)}(F(z)) = \sup_{1\le k \le n} \frac{\log (| f_k(z) |/a_k)}{\log(a_k / b_k)}\,, \]
and its relative capacity is 
\[ C(\clo{{\cal U}_b},{\cal U}_a)
= \frac{(2\pi)^n m_0}{\prod_{k=1}^{n} \log(a_k/b_k)}\,,\]
where $m_0$ is the multiplicity of $F$. 
\end{Proposition}

\begin{proof}
This result is an easy consequence of 
Proposition $4.5.14$ in \cite{K2} and Lemma $4.1$ in \cite{N4},
since the holomorphic mapping $F$ is proper and surjective from 
${\cal U}_a$ onto the open polydisc $P(O,a)$, and from 
${\cal U}_b$ onto the open polydisc $P(O,b)$, respectively.
\end{proof}

\subsubsection{A sequence of finite rank operators $(J_m)$ approximating $J$}
Let us recall that in order to find upper bounds for the Kolmogorov widths 
$d_m({\cal A}_K^D)$, it suffices to provide upper bounds for the approximation numbers 
$a_m(J)$ of the canonical identification
\[ \begin{array}{lcll}
J : & H^\infty(D) & \to  &A(K)\\
& f & \mapsto  &Jf=f|_{K}\,. 
\end{array} \]
As detailed in Subsection~\ref{sec:1.3}, this follows since 
\[ d_m({\cal A}_K^D)=d_m(J)\leq a_m(J) \quad (\forall m\in \mathbb{N})\,.\]
Specialising to the case where $(K,D)= (\clo{{\cal U}_b}, {\cal U}_a)$ is a special holomorphic 
polyhedral condenser we shall now construct  a sequence of finite rank operators 
$J_m : H^\infty(D)\to A(K)$ which approximate $J$ at a certain stretched exponential speed. 
The main tool will be the Bergman-Weil 
integral formula originally due to Weil \cite{Weil} and Bergman \cite{Bergman},  
which we turn to shortly.

Before doing so, we briefly recall multi-index notation. For 
$z\in \bbC^n$ and $\nu \in \bbZ^n$ with $z=(z_1,\ldots, z_n)$ and $\nu=(\nu_1,\ldots, \nu_n)$ we 
write  $z^\nu=\prod_{k=1}^n z_k^{\nu_k}\,.$ We use the symbol ${\cal I}$ to mean ${\cal I}=(1, \ldots, 1)$, 
so that $z^{\cal I}=\prod_{k=1}^n z_k\,.$

Suppose now that $F= (f_1,$ $\ldots, f_n)$ is the underlying mapping of the 
special holomorphic polyhedral condenser $(\clo{{\cal U}_b}, {\cal U}_a)$, that is, 
$F:\Omega \to \mathbb{C}^n$ is a 
holomorphic mapping on a pseudoconvex open set $\Omega$ in $\bbC^n$ containing 
${\cal U}_a$ and $\clo{{\cal U}_b}$. It turns out that there is a 
holomorphic function $G:\Omega \times \Omega \to \bbC^{n \times n}$ such that 
\[ 
F(\zeta)-F(z) = G(\zeta,z)(\zeta-z) \quad (\forall \zeta, \, z \in \Omega)\,.
\]
The existence of $G$ for pseudoconvex $\Omega$ is a non-trivial matter and is originally 
due to Hefer~\cite{Hefer} (see also \cite{Sha}, Paragraphs $30$ and $50$). 

If $\partial^*{\cal U}_a$ 
denotes the distinguished boundary of ${\cal U}_a$, that is 
\[ \partial^*{\cal U}_a= \all{z \in \Omega}{| f_k(z) | =a_k \mbox{ for all } k \in \{1,\ldots ,n\}}\,,
\]
considered as an $n$-dimensional surface with a suitable orientation, 
then the Bergman-Weil integral formula on ${\cal U}_a$ can be stated as follows (see 
\cite[Paragraph~30]{Sha}): for any $g\in H^\infty({\cal U}_a)$ and any 
$z\in {\cal U}_a$ we have 
\begin{equation*}
\label{eq:bergmanweilf}
g(z)=\frac{1}{(2\pi i)^n}\int_{\partial^*{\cal U}_a}^* g(\zeta) \frac{\det(G(\zeta,z))}{(F(\zeta)-
F(z))^{\cal I}}\,d\zeta\,,
\end{equation*}
where $d\zeta$ is the $n$-form $d\zeta_1\land \ldots\land d\zeta_n$. Here, 
the star at the top of the integral sign indicates that integration is to be taken over 
any $\partial^*{\cal U}_{a'}$ where 
$a' \in (0,\infty)^n$ with $a'_k < a_k$ for any $k$, and $a'$ is chosen such that $z \in {\cal U}_{a'}$. 
It is not difficult to see that the integral 
$\int_{\partial^*{\cal U}_a}^*$ does not depend on this choice. Note that in 
the case where 
$g\in A(\overline{{\cal U}_a})$ then $\int_{\partial^*{\cal U}_a}^*$ is in fact the classical integration 
over $\partial^*{\cal U}_a$.

An important property of the Bergman-Weil integral representation is that 
its kernel is holomorphic in $z$. This implies that we can write the canonical identification 
$J:H^\infty({\cal U}_a)\to A(\clo{{\cal U}_b})$ as an infinite series of operators, all of which, as we shall see later, are finite rank. More 
precisely, the following holds. For $g \in H^\infty({\cal U}_a)$ and $z\in 
\overline{{\cal U}_b}$, we have
\begin{align}
(Jg)(z)=g(z)
& = \frac{1}{(2\pi i)^n}\int_{\partial^*{\cal U}_a}^* g(\zeta) 
\frac{\det(G(\zeta,z))}{(F(\zeta)-F(z))^{\cal I}}\,d\zeta \nonumber \\
& = \frac{1}{(2\pi i)^n}\int_{\partial^*{\cal U}_a}^* g(\zeta)\frac{\det(G(\zeta,z))}{F(\zeta)^{\cal I}}  
\sum_{l=1}^\infty \frac{F(z)^{\nu(l)}}{F(\zeta)^{\nu(l)}}\,d\zeta \nonumber \\
& =\sum_{l=1}^\infty F(z)^{\nu(l)} \frac{1}{(2\pi i)^n}\int_{\partial^*{\cal U}_a}^* g(\zeta)
\frac{\det(G(\zeta,z))}{F(\zeta)^{{\cal I}+ \nu(l)}} \,d\zeta\,,
\label{eq:Jintrep}
\end{align}
where $\nu:\mathbb{N}\to \mathbb{N}_0^n$ could, in principle, be any bijection, but we shall fix 
it so 
as to facilitate obtaining effective bounds for the approximation numbers of 
$J:H^\infty({\cal U}_a)\to A(\clo{{\cal U}_b})$. 

In order to achieve this, let 
$\alpha \in (0,1)^n$ be given by $\alpha_k=b_k/a_k$ for 
$k\in \{1,\ldots, n\}$. Now choose $\nu:\mathbb{N}\to \mathbb{N}_0^n$ so that 
$m\mapsto \gamma_m:=\alpha^{\nu(m)}$ is monotonically decreasing. In other words, the bijection 
$\nu$ is chosen to provide a non-increasing rearrangement $(\gamma_m)_{m\in \bbN}$ 
of the set $\{\alpha^\nu : \nu \in \bbN^n_0\}$. 

As a first step towards bounding the approximation numbers $a_m(J)$ 
we need to bound the rate of decay of $(\gamma_m)_{m\in \bbN}$. For this we require 
the 
following auxiliary result, the short proof of which is adapted from the proof of 
\cite[Lemma~2.4]{BMV}. 

\begin{Lemma} \label{lem30}
Let  $\beta \in (0,\infty)^n$ and let $N_\beta:[0,\infty)\to \bbN_0$ denote the counting function  
\[ N_\beta(r)=\sharp \alll{\nu \in \bbN^n_0}{\sum_{k=1}^n \nu_k\beta_k\leq r}\,.
\]
 Then 
\begin{equation}
\label{eq10}
\frac{1}{n!}\frac{r^n}{\prod_{k=1}^n\beta_k} \leq N_\beta(r) 
\leq \frac{1}{n!}\frac{\left (r +\sum_{k=1}^n\beta_k\right )^n}{\prod_{k=1}^n\beta_k} 
\quad (\forall r\geq 0)\,. 
\end{equation}
\end{Lemma}
\begin{proof}
Fix $r\geq 0$ and define the following sets 
\[
\mathcal{N}_\beta(r)=
\alll{\nu \in \bbN^n_0}{\sum_{k=1}^n \nu_k\beta_k\leq r}\,,
\]
\[
\mathcal{S}_\beta(r)= 
\alll{\xi\in [0,\infty)^n}{\sum_{k=1}^n \xi_k\beta_k\leq r}\,,
\]
\[ \mathcal{C}_\nu = \all{\xi\in \mathbb{R}^n}{\nu_k\leq \xi_k \leq \nu_k+1\ \forall k\in 
\set{1,\ldots,n}  } \quad (\nu \in \mathbb{N}_0^n)\,,
\]
\[ \mathcal{M}_\beta(r)= \bigcup_{\nu \in \mathcal{N}_\beta(r)} \mathcal{C}_\nu\,.
\]
It is not difficult to see that we have the following inclusions
\[ \mathcal{S}_\beta (r) \subset \mathcal{M}_\beta(r) \subset 
\mathcal{S}_\beta (r+\sum_{k=1}^n\beta_k)\,,\]
from which the inequalities~(\ref{eq10}) readily follow, by computing the volume of the respective sets.  
 \end{proof}

\begin{Remark}
The lemma above implies that 
\[ N_\beta(r) \sim \frac{1}{n!}\frac{r^n}{\prod_{k=1}^n\beta_k} \text{ as $r\to \infty$.} \]
This asymptotic also follows easily from Karamata's Tauberian theorem (see, for example, 
Lemma~6.1 in \cite{LQR}).  
The lemma above, however, provides completely explicit bounds valid for all $r\geq 0$.  
\end{Remark}

We now have the following upper bound for $(\gamma_n)_{n\in \bbN}$. 

\begin{Lemma} \label{lem31}
Let $\alpha\in (0,\infty)^n$ and let $(\gamma_m)_{m\in \bbN}$ 
denote a non-increasing rearrangement of the set 
$\all{\alpha^\nu}{\nu \in \bbN^n_0}$. 
Writing 
\[ c=n ! \prod_{k=1}^n \log \alpha_k^{-1}
\] 
we have 
\begin{align*}
\gamma_m & \leq \frac{1}{\prod_{k=1}^n \alpha_k} \exp ( -(cm)^{1/n} )\quad (\forall \, m\in \bbN)\,, \\
\sum_{l=m+1}^\infty \gamma_l &     \leq \frac{1}{\prod_{k=1}^n \alpha_k \log \alpha_k^{-1}} 
\sum_{k=0}^{n-1} \frac{(c m)^{k/n}}{k!} \exp( -(cm)^{1/n})  \quad (\forall \, m\in \bbN_0)\,.
\end{align*}      
\end{Lemma}

\begin{proof}
We start by observing that 
\[ m \leq \sharp \{\nu \in \bbN_0^n : \alpha^\nu \geq \gamma_m\} 
= N_\beta(\log \gamma_m^{-1})\quad (\forall m \in \bbN)\,,
\]
where $N_\beta$ is the counting function from Lemma \ref{lem30} with 
$\beta\in (0,+\infty)^n$ given by 
$\beta_k=\log \alpha_k^{-1}$ for $1\leq k \leq n$. 
The first bound now follows from the upper bound in Lemma \ref{lem30}. 

For the second bound we use the first bound together with a majorisation 
of the sum by an integral to obtain 
\begin{align*}
\sum_{l=m+1}^\infty \gamma_l 
 & \leq \frac{1}{\prod_{k=1}^n\alpha_k}\sum_{l=m+1}^\infty \exp(-(cl)^{1/n}) \\
 & \leq \frac{1}{\prod_{k=1}^n\alpha_k}\int_{m}^\infty \exp(-(ct)^{1/n})\,dt \\
 & = \frac{1}{\prod_{k=1}^n\alpha_k}
    \left [  -\frac{n !}{c}\sum_{k=0}^{n -1} \frac{(ct)^{k/n}}{k!} \exp(-(ct)^{1/n} ) 
    \right ]_m^{\infty}\\
 & = \frac{1}{\prod_{k=1}^n\alpha_k\log \alpha_k^{-1}}
       \sum_{k=0}^{n -1} \frac{(cm)^{k/n}}{k!} \exp(-(cm)^{1/n} )
\end{align*}
and we are done.
\end{proof}

We are now ready to define the sequence of finite rank operators 
$J_m:H^\infty({\cal U}_a) \to A(\overline{{\cal U}_b})$ alluded to earlier. 
For $m \in \bbN$ and $\forall z\in \overline{{\cal U}_b}$ write 
\begin{align}
(J_m g)(z)
&= \sum_{l=1}^m F(z)^{\nu(l)} \frac{1}{(2\pi i)^n}\int_{\partial^*{\cal U}_a}^* g(\zeta)
\frac{\det(G(\zeta,z))}{F(\zeta)^{{\cal I}+ \nu(l)}} \,d\zeta \nonumber \\
&= \frac{1}{(2\pi i)^n}\int_{\partial^*{\cal U}_a}^* g(\zeta)\frac{\det(G(\zeta,z))}{F(\zeta)^{{\cal I}}}
\sum_{l=1}^m \frac{F(z)^{\nu(l)}}{F(\zeta)^{\nu(l)}} \,d\zeta \,.
\label{eq:Jmdef}
\end{align}
Clearly, $J_m$ is a well-defined operator from $H^\infty({\cal U}_a)$ 
to $A(\overline{{\cal U}_b})$. Moreover, we have the following upper bound for the rate at which the sequence $(J_m)_{m\in \mathbb{N}}$ approximates $J$. 

\begin{Lemma}
\label{lem:JminusJm}
Let $(\clo{{\cal U}_b}, {\cal U}_a)$ be a special holomorphic polyhedral condenser in 
$\mathbb{C}^n$ with underlying mapping $F$. Then 
\begin{equation*}
\| J-J_m \|_{H^\infty({\cal U}_a)\to A(\overline{{\cal U}_b})}\leq C 
\left( \sum_{k=0}^{n-1}\frac{(c m)^{k/n}}{k!}\right) \exp(-(c m)^{1/n}) 
\quad (\forall m\in \mathbb{N})\,,
\end{equation*}
where 
\[ c=n! \prod_{k=1}^n \log \alpha_k^{-1}\,, \]  
\begin{equation}
\label{eq:constantC}
C= [(2\pi)^n\prod_{k=1}^n \alpha_k \log \alpha_k^{-1}]^{-1} \sup_{z\in \overline{{\cal U}_b}} 
\int_{\partial^*({\cal U}_a)} \frac{| \det(G(\zeta,z)) |}
{ | F(\zeta)^{\cal I} |}\,| d\zeta |\,,
\end{equation}
and, as before, $\alpha_k=b_k/a_k$ for $1\leq k \leq n$.  
\end{Lemma}
\begin{proof}
Fix $z\in \clo{{\cal U}_b}$. 
Recall that the integration $\int_{\partial^*{\cal U}_a}^*$ 
in (\ref{eq:Jmdef}) is in fact an integration 
over any $\partial^*{\cal U}_{a'}$ where 
$a' \in (0,\infty)^n$ with $b_k < a'_k < a_k$ for any $k$ is chosen so that 
$z\in {\cal U}_{a'}$. 
Thus, for $g\in H^\infty({\cal U}_a)$ with $\|g \|_{H^\infty({\cal U}_a)} \leq 1$, 
we have using the calculation in (\ref{eq:Jintrep})
\begin{multline*}
| (Jg-J_m g)(z) | \leq \\
\left( \frac{1}{(2\pi)^n} \int_{\partial^*({\cal U}_{a'})} \frac{| \det(G(\zeta,z)) |}{ | F(\zeta)^{\cal I} |}\, 
| d\zeta | \right)
\sum_{l=m+1}^\infty \left(\frac{b_1}{a_1'}\right)^{\nu_1(l)} \cdots 
\left(\frac{b_n}{a_n'}\right)^{\nu_n(l)}\, .
\end{multline*}
Since there exists $\delta >0$ such that $b_k < b_k + \delta \le a'_k < a_k$ for any $k$, 
the power series 
$\sum_{l=m+1}^\infty \left(\frac{b_1}{a_1'}\right)^{\nu_1(l)} 
\cdots \left(\frac{b_n}{a_n'}\right)^{\nu_n(l)}$ converges to $\sum_{l=m+1}^\infty \gamma_l$ 
when $a'$ tends to $a$.
Consequently we obtain
\[ | (Jg-J_m g)(z) | \leq \left(\frac{1}{(2\pi)^n} \int_{\partial^*({\cal U}_a)} 
\frac{| \det(G(\zeta,z)) |}{| F(\zeta)^{\cal I} |}
| d\zeta | \right) \sum_{l=m+1}^\infty\gamma_l 
\quad (\forall z\in \overline{{\cal U}_b})\,,
\]
where we have used the fact that 
\[ a'\mapsto \frac{1}{(2\pi)^n} \int_{\partial^*({\cal U}_{a'})} 
\frac{| \det(G(\zeta,z)) |}{ | F(\zeta)^{\cal I} |}\, | d\zeta | \]
is continuous at $a$. 
Thus using Lemma \ref{lem31} we have for any $m \in \bbN$
\begin{equation}
\| J-J_m \|_{H^\infty({\cal U}_a)\to A(\overline{{\cal U}_b})}\leq C 
\left( \sum_{k=0}^{n-1}\frac{(c m)^{k/n}}{k!}\right) \exp(-(c m)^{1/n}) \,,
\end{equation}
as claimed. 
\end{proof}

\subsubsection{Upper bound for the rank of $J_m$}
For $(\clo{{\cal U}_b}, {\cal U}_a)$ a special holomorphic polyhedral condenser with underlying 
mapping $F$ we want to use the previous Lemma~\ref{lem:JminusJm} to obtain an upper bound 
for the approximation numbers of $J:H^\infty({\cal U}_a) \to A(\clo{{\cal U}_b})$. For this we 
need to impose an extra assumption on $F$, and hence on the condenser: we shall call 
$(\clo{{\cal U}_b}, {\cal U}_a)$ \emph{non-degenerate} if $O$ is a regular value of the underlying 
proper mapping $F$ from ${\cal U}_a$ onto $P(O,a)$. Since $F$ has finite multiplicity $m_0$ this 
implies that $F|_{{\cal U}_a}$ has exactly $m_0$ distinct zeros. 

Note that this is no essential restriction for our purposes, since if $O$ is not a regular value of $F$, 
then we can replace $F$ by $F-c$, where $c$ is a regular value for $F$, which can be chosen 
arbitrarily small. 

\begin{Lemma}
\label{lem:rankJm}
Let $(\clo{{\cal U}_b}, {\cal U}_a)$ be a non-degenerate 
special holomorphic polyhedral condenser with underlying mapping $F$ 
and $J_m$ the operator defined in (\ref{eq:Jmdef}). Then 
\[ {\rm rank}(J_m)\leq m_0m \quad (\forall m\in \mathbb{N})\,, \]
where $m_0$ is the multiplicity of $F$. 
\end{Lemma}

\begin{proof} 
Since the condenser is non-degenerate the underlying mapping $F$ has exactly $m_0$ 
distinct zeros, call them $z^{(1)},\ldots,z^{(m_0)}$. 
Next choose $c=(c_1,\ldots,c_n)$ where each $c_j$ is a positive 
real number small enough so that ${\cal U}_c$ has $m_0$ connected components 
${\cal U}_c^k$, the closures of which are pairwise 
disjoint. 

Now fix $z \in {\cal U}_c$ and let ${\cal U}_c^{k_0}$ denote the neighbourhood of $z^{(k_0)}$ to 
which $z$ belongs.  For any $m\in \mathbb{N}$ and any $g \in H^\infty({\cal U}_a)$, we have 
\begin{align*}
(J_mg)(z) 
& = \sum_{l=1}^m F(z)^{\nu(l)} \frac{1}{(2\pi i)^n}\int_{\partial^*{\cal U}_a}^* g(\zeta)
\frac{\det(G(\zeta,z))}{F(\zeta)^{{\cal I}+ \nu(l)}} \,d\zeta\\
& = \sum_{k=1}^{m_0} \sum_{l=1}^m F(z)^{\nu(l)} \frac{1}{(2\pi i)^n} \int_{\partial^*{\cal U}_c^k} 
g(\zeta)\frac{\det(G(\zeta,z))}{F(\zeta)^{{\cal I}+ \nu(l)}} \,d\zeta \,.
\end{align*}
Each integral 
$\displaystyle \int_{\partial^*{\cal U}_c^k} g(\zeta)
\frac{\det(G(\zeta,z))}{F(\zeta)^{{\cal I}+ \nu(l)}} \,d\zeta$ is over a small 
neighbourhood of a zero $z^{(k)}$ of $F$. 
Since $z \in {\cal U}_c^{k_0} \subset \overline{{\cal U}_c^{k_0}}$, 
which is disjoint from $\overline{{\cal U}_c^k}$  for any $k\neq k_0$, we have 
\[ \int_{\partial^*{\cal U}_c^k} g(\zeta) \frac{\det(G(\zeta,z))}{(F(\zeta)-F(z))^{\cal I}}\,d\zeta =0,
\quad (k\neq k_0)\,.\]
For any $1\le k \le m_0$, the underlying mapping 
$F$ is a biholomorphism from ${\cal U}_c^{k}$,  a neighbourhood 
of $z^{(k)}$, onto a neighbourhood $V_{k}$ of the origin. 
Let $\pi_{k}$ denote the corresponding inverse mapping, so that 
$F(\zeta)=w$ implies $\pi_k(w)=\zeta$. Then for any $l \ge 1$ 
and for any $1\le k \le m_0$, we have 
\begin{align*}
& \frac{1}{(2\pi i)^n}\int_{\partial^*{\cal U}_c^{k}} g(\zeta)
  \frac{\det(G(\zeta,z))}{F(\zeta)^{{\cal I}+ \nu(l)}} \,d\zeta \\
= & \frac{1}{(2\pi i)^n}\int_{\partial^*P(O,c)} g(\pi_k(w))
   \frac{\det(G(\pi_k(w),z))}{w^{{\cal I}+ \nu(l)}} \,(\Jac \, \pi_k)(w)dw \\
= & \frac{1}{\nu(l) !}  \frac{\partial^{\nu(l)}}{\partial w^{\nu(l)}}\left( g\circ \pi_{k} \cdot   
    \det(G(\pi_k,z)) \cdot  \Jac \, \pi_k \right)(0)\, ,
\end{align*}
where $\Jac \, \pi_k$ is the complex Jacobian of the holomorphic map $\pi_k$: 
\[ \Jac \, \pi_k=\det \left( \frac{\partial \pi_{k,j}}{\partial w_{j'}} \right)_{1\le j, j'\le n}\,.\]
Consequently we have for any $z \in {\cal U}_c$
\[ (J_m g)(z)=\sum_{l=1}^m \sum_{k=1}^{m_0} \frac{1}{\nu(l) !}  \frac{\partial^{\nu(l)}}{\partial 
w^{\nu(l)}}\left( g\circ \pi_{k} \cdot \det(G(\pi_k,z)) \cdot \Jac \, \pi_k \right)(0)
F(z)^{\nu(l)} \,.
\]
But since the 
right-hand side of this equality is a well defined holomorphic function on all of ${\cal U}_a$ and  
since $J_m g$ is in ${\cal A}(\overline{{\cal U}_b})$, the analytic continuation principle 
implies that the equality above holds for every $z\in \clo{{\cal U}_b}$.

Using the product rule, it follows that the partial derivative 
\[ 
\frac{\partial^{\nu(l)}}{\partial w^{\nu(l)}}\left( g\circ \pi_{k} \cdot 
  \det(G(\pi_k,z))\cdot  \Jac \, \pi_k \right)(0)
\] 
involves only partial derivatives of $g$ of the form  
$\frac{\partial^\mu g}{\partial \zeta^\mu}(z^{(k)})$,
where the multi-indices satisfy 
$0\le \mu_1 \le \nu_1(l), \ldots, 0 \le \mu_n \le \nu_n(l)$. Due to our choice of $\nu$, 
this implies that these multi-indices are of the form 
$\nu(l')$ with $1\le l' \le m$.

As a result, the collection of values  
\[ \displaystyle \frac{\partial^{\nu(l)} g}{\partial \zeta^{\nu(l)}}(z^{(k)})
\quad (1\le k \le m_0,\,1\le l \le m) \] 
completely determine the function $J_m(g)$ on $\overline{{\cal U}_b}$. 
Thus, for each $m\in \mathbb{N}$, the operator $J_m$ has finite rank 
with its rank bounded above by~$m_0m$.
\end{proof}

\subsubsection{Upper bound for Kolmogorov widths in the case of  special holomorphic polyhedra.}
We are now able to state and prove the main result of this subsection: 
explicit and asymptotically sharp upper bounds for the Kolmogorov widths of special 
holomorphic polyhedral condensers. 
\begin{Proposition}
\label{prop:specialcondenser}
Let $(\clo{{\cal U}_b}, {\cal U}_a)$ be a non-degenerate 
special holomorphic polyhedral condenser in $\mathbb{C}^n$ and 
$J:H^\infty({\cal U}_a)\to A(\clo{{\cal U}_b})$ the canonical identification. Then 
for any $m\in \mathbb{N}$ with $m>m_0$ we have 
\[ 
d_m(J)\leq 
a_m (J) \le C \left( \sum_{k=0}^{n-1}\frac{1}{k!}\left(c \frac{m-m_0}{m_0}\right)^{k/n}\right) 
\exp\left(-\left(c \frac{m-m_0}{m_0}\right)^{1/n}\right) \]
where 
\[ c=n! \prod_{k=1}^n \log(a_k/b_k)\,, \] 
$m_0$ is the multiplicity of the underlying mapping of the condenser, and $C$ is the explicit 
constant given in (\ref{eq:constantC}). In particular, we have 
\begin{equation}
\label{eq:specialcondenserasymp}
\liminf_{m\to \infty} -\frac{\log d_m\left({\cal A}_{\overline{{\cal U}_b}}^{{\cal U}_a}\right)}{m^{1/n}} 
\ge  
2\pi \left( \frac{ n!}{C(\overline{{\cal U}_b}, {\cal U}_a)}\right )^{1/n}\, .
\end{equation}
\end{Proposition}

\begin{proof}
The explicit upper bound follows from Lemma~\ref{lem:JminusJm} and Lemma~\ref{lem:rankJm} 
together with the fact  the sequence $m\mapsto a_m(J)$ is monotonically decreasing.  
Moreover, 
using Proposition~\ref{theor30}, which yields 
\[ C(\overline{{\cal U}_b}, \, {\cal U}_a) = \frac{(2\pi)^n m_0}{\prod_{k=1}^{n} \log(a_k/b_k)}\,,\] 
the assertion (\ref{eq:specialcondenserasymp}) follows. 
\end{proof}

\begin{Remark}
A stronger version of the asymptotics (\ref{eq:specialcondenserasymp}) above, 
with similar hypotheses but with a limit instead of the limit inferior and equality instead of inequality 
can be found, without proof, as Proposition~5.1 in Zakharyuta's survey article \cite{Z11}, where 
it is credited to \cite{Z4} in which it appears, again without proof, as a consequence of 
\cite[Theorem~4.5]{Z4}.
\end{Remark}

\subsection{Exhaustion of $D$ and $K$ by special holomorphic polyhedra}
We shall now extend the sharp asymptotic upper bound obtained at the end of the previous 
subsection for special polyhedral condensers to more general condensers $(K,D)$ where  
$D$ is strictly hyperconvex and $K$ a regular compact subset of $D$. 

In this setup, we have the following refinement of Theorem $3$ in \cite{N4},
itself a quantitative version of a lemma of Bishop~\cite{Bi}, which 
provides an external approximation of $\hat{K}_D$, the holomorphically convex hull of 
$K$ in $D$, and an internal approximation of $D$ 
by two special holomorphic polyhedra defined simultaneously 
by the same holomorphic mapping such
that the relative capacities of the approximations converge to the capacity 
$C(K,D)=C(\hat{K}_D,D)$. 

\begin{Theorem} \label{thm:Capprox}
Let $D$ be a strictly hyperconvex domain in $\mathbb{C}^n$ and $K$ a regular compact 
subset of $D$. Then, for every $j\in \mathbb{N}$, there is a 
non-degenerate special holomorphic polyhedral 
condenser $(\clo{{\cal U}}^1_j, {\cal U}^2_j)$ with 
\[ K\subset \hat{K}_D \subset \clo{{\cal U}}^1_j \subset {\cal U}^2_j \subset D \]
such that 
\[ \lim_{j\to \infty} C(\clo{{\cal U}}^1_j, {\cal U}^2_j) = C(K,D)\,.\]
\end{Theorem}

\begin{proof}
We start by recalling some notation already used in 
Subsection~\ref{sec:1.1}, as well as in 
Subsection~\ref{sec:2.4}, immediately above Lemma~\ref{lem:10}. 
  
Since $D$ is strictly hyperconvex there exists a bounded domain $\Omega$ and an exhaustion 
function $\varrho \in  PSH( \Omega, (-\infty, 1))\cap C(\Omega)$ 
such that $D=\all{z\in\Omega}{\varrho (z)<0}$ and, for all real numbers $c \in [0,1]$, the 
open set $\all{z \in \Omega}{ \varrho (z)<c }$ is connected.

For any integer $j \ge 1$, let $D_j$ denote the bounded hyperconvex domain 
$\all{z\in\Omega}{\varrho (z)<1/j}$.

Since $K$ is a regular compact subset of $D$ the relative extremal function
$u_{K,D}$ is continuous in $\overline{D}$. 
Let now $u$ and $u_j$ denote the relative extremal
functions $u_{K,D}$ and $u_{K,D_j}$, respectively.

Lemma~2.6 in \cite{N4} then implies the following: 
$K$ is regular for any $D_j$ with $j\ge 1$, the sequence $(u_j)$ converges uniformly on 
$\clo{D}$ to $u$, and the increasing sequence
of capacities $(C(K,D_j))_{j\in \mathbb{N}}$ converges to the capacity $C(K,D)$.

For any $r\in [-1,0]$ write $D(r)$ for the set $\all{z \in D}{u(z) < r}$. 
As $u$ is continuous on $\clo{D}$ and an exhaustion function for $D$, each $D(r)$ is open. 
Moreover, for any  $-1<r<0$ sufficiently close to $0$, 
the open sets $D(r)$ are connected, since $D$ is, and provide an internal exhaustion of $D$. 
Similarly, for any $-1<r<0$ sufficiently near $-1$, the open sets $D(r)$ provide an external 
exhaustion of $\hat{K}_D = \all{z\in D}{u_{K,D}(z)=-1}$.

We shall now proceed to the construction of the special holomorphic polyhedra by a five-step 
process of definitions and  
assertions, the first two of which, follow from \cite[Theorem~3]{N4}. 

For any $\epsilon, \epsilon'> 0$ sufficiently small with  
$\epsilon'<\epsilon$,  there exist two integers $j \ge 2$ and $p\ge 2$ and there
exist $n$ holomorphic functions $f_1,\ldots, f_n \in {\cal O}(D_j)$ such that the following assertions 
hold. \\
(i) $\displaystyle \frac{1}{p} \log | f_l(z) | \le u_{2j}(z)$ in $D_{2j}$, for all $1\le l \le n$.\\
(ii) There exist two special holomorphic polyhedra 
  ${\cal U}^0$ and ${\cal U}^2$ and a positive constant $\beta(\epsilon')$ satisfying 
$\beta(\epsilon') \le \epsilon'/2$, such that
\[ K \subset \hat{K}_D \subset \overline{D(-1+\epsilon)}\subset {\cal U}^0
\subset D( -1+\epsilon + \epsilon') \]
\[ \mbox{and }\,  \overline{D(-\epsilon)}
\subset {\cal U}^2 \subset D( -\epsilon + \epsilon').\]
Here, ${\cal U}^0$ is the finite union of the connected components of the open set
\[ \all{z \in D}{\sup_{1\le l \le n} \frac{1}{p} \log | f_l(z) | < -1 +
\epsilon +\beta(\epsilon')}\]
that meet $\overline{D(-1+\epsilon)}$,  and ${\cal U}^2$ is the connected component containing $
\overline{D(-\epsilon)}$, of the open set 
\[ \all{z \in D}{\sup_{1\le l \le n}\frac{1}{p} \log | f_l(z) | < -\epsilon + \beta(\epsilon')}\,.\]
Using the inclusions in assertion (ii), we obtain the next assertion\\
(iii) $\,  C(\overline{D(-1+\epsilon)},D( -\epsilon + \epsilon')) 
\le C(\overline{{\cal U}}^0,{\cal U}^2) 
\le C(\overline{D(-1+\epsilon+\epsilon')},D(-\epsilon))\,.$
Indeed, we have 
\[ C(\overline{D(-1+\epsilon)},D( -\epsilon + \epsilon')) 
\, \le \, C(\overline{\cal U}^0,D( -\epsilon + \epsilon')) 
\le \, C(\overline{\cal U}^0,{\cal U}^2)\]
and 
\[C(\overline{\cal U}^0,{\cal U}^2)
\le \, C(\overline{D(-1+\epsilon+\epsilon')},{\cal U}^2) 
\le \, C(\overline{D(-1+\epsilon+\epsilon')},D(-\epsilon))\,.\]
(iv) Denote by ${\cal U}^1$ the special holomorphic polyhedron 
defined as the finite union of all connected components of the open set
\[ \all{z \in D}{\sup_{1\le l \le n} \frac{1}{p} \log | f_l(z) | < -1 +
\epsilon +\beta(\epsilon')}\] 
included in ${\cal U}^2$.\\
 (v) The holomorphic mapping $$F=(f_1,...,f_n) : {\cal U}^2 \longrightarrow \bbC^n$$
is proper and surjective from the bounded special holomorphic polyhedron
${\cal U}^2$ 
to the polydisc $P(0,r_2)$ 
and from the bounded special holomorphic polyhedron ${\cal U}^1$ to 
the polydisc $P(0,r_1)$, where $r_2=\exp(p(-\epsilon + \beta(\epsilon')))$ 
and $r_1=\exp(p(-1+\epsilon + \beta(\epsilon')))$), respectively. 
This follows from \cite[Proposition~3.4]{N4}. 

All in all, the above yields, for any $\epsilon, \epsilon'> 0$ sufficiently small with  
$\epsilon'<\epsilon$, three special holomorphic polyhedra 
${\cal U}^0_{\epsilon, \epsilon'}$, 
${\cal U}^1_{\epsilon, \epsilon'}$, and 
${\cal U}^2_{\epsilon, \epsilon'}$ with 
\[ K\subset \hat{K}_D\subset {\cal U}^0_{\epsilon, \epsilon'}\subset {\cal U}^1_{\epsilon, \epsilon'}
\subset {\cal U}^2_{\epsilon, \epsilon'}\subset D\,. \]
By (iii) we know that $C(\overline{{\cal U}}^0_{\epsilon, \epsilon'},{\cal U}^2_{\epsilon, \epsilon'})$
converges to $C(\hat{K}_D,D)=C(K,D)$ when $\epsilon$ and $\epsilon'$ both tend to $0$.
Moreover, using an argument similar to the proof of Lemma 5.2 in \cite{N4} it follows that 
$C(\overline{{\cal U}}^1_{\epsilon, \epsilon'},{\cal U}^2_{\epsilon, \epsilon'})$ 
also converges to $C(K,D)$ when $\epsilon$ and $\epsilon'$ both tend to $0$. This, coupled with 
assertion (v) yields the desired sequence of special holomorphic polyhedral condensers, each of 
which can be chosen to be non-degenerate, by possibly replacing the underlying 
mapping $F_{\epsilon, \epsilon'}$ with $F_{\epsilon, \epsilon'}-c_{\epsilon, \epsilon'}$, for 
$c_{\epsilon, \epsilon'}$ a sufficiently small regular value of $F_{\epsilon, \epsilon'}$. 
\end{proof}

Combining the previous theorem with the main result in the previous subsection 
we now have the following. 
\begin{Proposition}
\label{prop:312}
Let $D$ be a strictly hyperconvex domain in $\mathbb{C}^n$ and $K$ a regular 
compact subset of $D$. Then we have the asymptotics
\[ \liminf_{m\to \infty} -\frac{\log d_m\left({\cal A}_{K}^{D}\right)}{m^{1/n}} \ge (2\pi) \left(\frac{ n!}
{C(K,D)}\right)^{1/n}\, .\]
\end{Proposition}

\begin{proof}
Let $(\clo{{\cal U}}^1_j, {\cal U}^2_j)_j$ 
denote the sequence of non-degenerate special holomorphic polyhedral 
condensers given in Theorem~\ref{thm:Capprox}. Since 
\[ K\subset \hat{K}_D \subset \clo{{\cal U}}^1_j \subset {\cal U}^2_j \subset D\,, \]
we have, for every $j\in \mathbb{N}$ 
\[ H^\infty(D)\hookrightarrow H^\infty( {\cal U}^2_j) \to A(\clo{{\cal U}}^1_j) \to 
A(K)\,. \]
Both $H^\infty(D)\hookrightarrow H^\infty( {\cal U}^2_j)$ and  
$A(\clo{{\cal U}}^1_j) \to A(K)$ are easily seen to be continuous with norm $1$, so
\[ d_m({\cal A}_K^D) \le d_m({\cal A}_{\overline{\cal U}^1_j}^{{\cal U}^2_j}) \quad 
(\forall m,j\in \mathbb{N})\,, \]
and hence 
\[ 
\liminf_{m\to \infty} -\frac{\log d_m\left({\cal A}_{K}^{D}\right)}{m^{1/n}} \ge  \liminf_{m\to \infty} -
\frac{\log d_m\left({\cal A}_{\overline{{\cal U}}^1_j}^{{\cal U}^2_j}\right)}{m^{1/n}} 
\quad (\forall j\in \mathbb{N})\,.
\]
Using Proposition~\ref{prop:specialcondenser} we have 
\[ 
\liminf_{m\to \infty} -\frac{\log d_m\left({\cal A}_{\overline{{\cal U}}^1}^{{\cal U}^2}\right)}{m^{1/n}} 
\ge  (2\pi) \left (\frac{ n!} {C(\overline{{\cal U}}^1_j, {\cal U}^2_j)}\right)^{1/n}
\quad (\forall j \in \mathbb{N})
\]
and the assertion follows as 
$C(\overline{{\cal U}}^1_j, {\cal U}^2_j)$ converges to $C(K,D)$ as $j\to \infty$ by 
Theorem~\ref{thm:Capprox}. 
\end{proof}

The hypothesis in the proposition above can be weakened further using the following result. 

\begin{Lemma} \label{lem32}
Let $D$ be a bounded hyperconvex domain in $\bbC^n$ containing a 
compact subset $K$.  
Then we can approximate $\hat{K}_D$ externally and $D$ internally 
by two sequences $(K_j)$ and $(D_j)$ such that:
\begin{itemize}
\item[(i)] $(K_j)$ is a decreasing sequence of compact sets and $(D_j)$ is an increasing sequence 
of domains with $\cap_j K_j = \hat{K}_D$, $K_{j+1} \subset \stackrel{\circ}{K}_{j}$, 
$\overline{D}_j\subset D_{j+1}$ and 
$\cup_j D_j = D$; 
\item[(ii)] for any $j$, each compact 
$K_j$ is holomorphically convex and regular in $D_j$, which in turn is strictly hyperconvex;
\item[(iii)] 
the sequence $(C(K_j,D_j))_j$ converges to $C(K,D)$.
\end{itemize}
\end{Lemma}

\begin{proof}
Note that $K$ is not supposed to be regular, that is, $u=u_{K,D}$ 
is not necessarily continuous in $\bar{D}$. 
Now, for any  $\epsilon > 0$ sufficiently small, 
let $K^\epsilon$ denote the compact subset of $D$ defined by 
\[ K^\epsilon = \all{z \in D}{{\rm dist}(z,\partial D ) \le \epsilon}\,.\]
Using Corollaries $4.5.9$, $4.5.11$ and Proposition $4.7.1$ in \cite{K2}, we deduce, successively,  the following:
\begin{itemize}
\item[(a)] $u_{K^\epsilon,D}$ is continuous in $D$, that is, 
$K^\epsilon$ is regular for $D$; 
\item[(b)]  $(u_{K^\epsilon,D})_{\epsilon >0}$ 
converges pointwise and monotonically from below to $u$ in $D$ when 
$\epsilon$ decreases to $0$; 
\item[(c)] $u_{K^\epsilon,D} = u_{\widehat{K^\epsilon}_D,D}$ and 
  $C(K^\epsilon,D)=C(\widehat{K^\epsilon_D},D)$; 
\item[(d)] $\widehat{K^\epsilon_D}$ approximates $\hat{K}_D$ in the sense that   
$\cap_{\epsilon >0} \widehat{K^\epsilon_D} = \hat{K}_D$; 
\item[(e)] $\lim_{\epsilon \to 0} C(K^\epsilon,D) = C(K,D)$.
\end{itemize}
Now, since $u_{K^\epsilon,D}:=u_\epsilon$ is continuous on $\overline{D}$ and is an 
exhaustion 
function for $D$, we can construct an internal exhaustion of $D$ by strictly 
hyperconvex domains as follows: 
for any $\delta > 0$ sufficiently small, $D_\epsilon(-\delta)=\{z \in D : u_\epsilon(z) < -
\delta\}$ is a 
strictly hyperconvex domain and $\cup_{\delta>0} D_\epsilon(-\delta) = D$.

We now choose $K_j=\overline{D_{\epsilon_j}(-1+\delta_j)}$ and 
$D_j=D_{\epsilon_j}(-\delta_j)$, where each $\epsilon_j, \delta_j >0$ is sufficiently small 
with the sequences $(\epsilon_j)_j$ and $(\delta_j)_j$ strictly decreasing to zero and 
$\delta_{j+1}$ is chosen such that 
\[ 
\sup\alll{u_{\epsilon_{j+1}}(z)}{z\in \overline{D}_{\epsilon_j}(-\delta_j)}<-\delta_{j+1}\,.\]
In this case the 
relative extremal function $u_{K_j,D_j}$ for $K_j$ and $D_j$, is easily seen to be given 
explicitly by 
\[ u_{K_j,D_j}(z)=\max\{\frac{u_{\epsilon_j}(z)+\delta_j}{1-2\delta_j},\, -1\} \quad 
(\forall z\in D_{\epsilon_j}(-\delta_j))\,,\]
while the corresponding relative capacity satisfies  
\[ C(K_j,D_j)=\frac{C(K^{\epsilon_j},D)}{(1-2\delta_j)^n}\,,\]
and the assertions of the lemma follow. 
\end{proof}

All in all, we now obtain the main result of this section, furnishing the second part of our main 
result, Theorem~\ref{theor10}. 

\begin{Theorem} 
\label{thm:mainupper}
Let $D$ be a bounded hyperconvex domain in $\mathbb{C}^n$ and $K$ a compact subset of $D$. 
Then 
\[ \liminf_{m\to \infty} -\frac{\log d_m\left({\cal A}_{K}^{D}\right)}{m^{1/n}} \ge (2\pi) \left(\frac{ n!}
{C(K,D)}\right)^{1/n}\, .
\]
\end{Theorem}
\begin{proof}
Let $(K_j)$ and $(D_j)$ denote the sequences of sets furnished by Lemma~\ref{lem32}. Then, for 
every  $j$, the set $D_j$ is a strictly hyperconvex domain in $\mathbb{C}^n$ and $K_j$ a regular 
holomorphically convex subset of $D_j$ with 
\[ K\subset \hat{K}_D \subset K_j \subset D_j \subset D\, .\] 
Thus, we have 
\[ H^\infty(D)\hookrightarrow H^\infty( D_j) \hookrightarrow A(K_j) \to 
A(K)\,. \]
Both $H^\infty(D)\hookrightarrow H^\infty( D_j)$ and  
$A(K_j) \to A(K)$ are easily seen to be continuous with norm $1$, so
\[ d_m \left({\cal A}_{K}^{D}\right ) \le d_m\left({\cal A}_{K_j}^{D_j}\right)  \quad 
(\forall m,j\in \mathbb{N})\,, \]
and we deduce
\[ \liminf_{m\to \infty} -\frac{\log d_m\left({\cal A}_{K}^{D}\right)}{m^{1/n}} \ge  \liminf_{m\to \infty} -
\frac{\log d_m\left({\cal A}_{K_j}^{D_j}\right)}{m^{1/n}} \quad (\forall j\in \mathbb{N}) \, .\]
But by Proposition~\ref{prop:312} we have 
\[ \liminf_{m\to \infty} -\frac{\log d_m\left({\cal A}_{K_j}^{D_j}\right)}{m^{1/n}} \ge  
(2\pi) \left(\frac{ n!} {C(K_j, D_j)}\right)^{1/n}\,,\]
and since by Lemma~\ref{lem32} 
the relative capacities $C(K_j,D_j)$ converge to $C(K,D)$ when $j$ tends to infinity, we finally 
deduce that
\[ \liminf_{m\to \infty} -\frac{\log d_m\left({\cal A}_{K}^{D}\right)}{m^{1/n}} \ge (2\pi) \left(\frac{ n!}
{C(K,D)}\right)^{1/n}\, ,\]
which concludes the proof. 
\end{proof}

 \section*{Acknowledgements}
St\'ephanie Nivoche would like to thank Bo Berndtsson for many inspiring discussions during the 
preparation of this article.

\end{document}